\documentclass[12pt,longbibliography]{article}

\usepackage[utf8]{inputenc}

\usepackage[T1]{fontenc}

\usepackage[a4paper, margin=2.7cm]{geometry}

\usepackage{amsmath, amsthm, amsfonts, amssymb}
\usepackage{mathrsfs}           

\usepackage[colorlinks=true,linkcolor=blue,citecolor=blue,urlcolor=blue,breaklinks]{hyperref}

\usepackage{bbm} 

\usepackage{dsfont}

\usepackage{breakurl}
\usepackage{natbib}
\usepackage{url}
\def\R{\mathbb{R}}

\def\1{\mathbbm{1}}

\def\d{\,\mathrm{d}}

\def \ddt{\frac{\mathrm{d}}{\mathrm{d}t}}
\def \ddt{\frac{\mathrm{d}}{\mathrm{d}t}}

\def\p{\partial}
\def\:{\colon}

\newtheorem{thm}{Theorem}[section]

\newtheorem{lem}[thm]{Lemma}
\newtheorem{prp}[thm]{Proposition}

\theoremstyle{definition}
\newtheorem{dfn}[thm]{Definition}
\theoremstyle{remark}
\newtheorem{rem}[thm]{Remark}
\theoremstyle{example}

\hypersetup{pdftitle={Asymptotic behaviour of neuron population models structured by elapsed-time}}
\hypersetup{pdfauthor={José A. Cañizo \& H. Yoldaş}}

\title{Asymptotic behaviour of neuron population models structured by
  elapsed-time}

\author{José A. Cañizo \and Havva Yoldaş}

\AtEndDocument{\vspace{2cm}{\footnotesize
  \noindent
  \textsc{José A. Cañizo. Departamento de
    Matemática Aplicada, Universidad de Granada, 18071 Granada,
    Spain.}
  \textit{Email address}: \texttt{\href{mailto:canizo@ugr.es}{canizo@ugr.es}}
  \par
  \addvspace{\medskipamount}
  \noindent
  \textsc{Havva Yoldaş. Basque Center for Applied Mathematics, Alameda de Mazarredo 14,
    48009 Bilbao, Spain \& Departamento de
    Matemática Aplicada, Universidad de Granada, 18071 Granada,
    Spain.}
  \textit{Email address}: \texttt{\href{mailto:hyoldas@bcamath.org}{hyoldas@bcamath.org}}
}}

\begin{document}

\maketitle

\begin{abstract}
  We study two population models describing the dynamics of
  interacting neurons, initially proposed by \citet*{PPS10,PPS14}. In
  the first model, the structuring variable $s$ represents the time
  elapsed since its last discharge, while in the second one neurons
  exhibit a fatigue property and the structuring variable is a generic
  ``state''. We prove existence of solutions and steady states in the
  space of finite, nonnegative measures. Furthermore, we show that
  solutions converge to the equilibrium exponentially in time in the
  case of weak nonlinearity (i.e., weak connectivity). The main
  innovation is the use of Doeblin's theorem from probability in order
  to show the existence of a spectral gap property in the linear
  (no-connectivity) setting. Relaxation to the steady state for the
  nonlinear models is then proved by a constructive perturbation
  argument.
\end{abstract}

\tableofcontents

\section{Introduction}
\label{sec:intro}

Several mean-field models have been proposed to understand the
electrical activity of a group of interacting neurons. They are all
based on simplified models for the electrical activity of a single
neuron, which give rise to an averaged partial differential equation
(PDE) or integro-differential equation suitable when the number of
neurons involved is large. In a rough approximation, neurons are
assumed to undergo some sort of natural ``charging'' process, with a
sudden ``discharge'' taking place in a stochastic way depending on the
current charge, the time since the last discharge, and the activity of
other connected neurons. A family of these models is structured by the
membrane potential of neurons; that is, the quantity under study is
the density $n(t,v)$ of neurons with potential $v$ at time $t$. The
mathematical theory of these models is recent; see
\citet{BookNeurobiologyv2, BrunelHakim4, Brunel2000, Brette2, Rossant,
  Caceres2011Analysis, Carrillo2013Classical, Caceres2016Blowup, article}.
The family of models we study in this paper is structured by the
elapsed time since the last discharge and has been proposed in
\citet{PPS10,PPS13,PPS14}. The first model is based on stochastic
simulations done in \cite{D12PHAM1998415}. It is a nonlinear version
of the conservative renewal equation (sometimes called as
McKendrick-Von Foerster equation) which has been well-studied by many
authors in the past as a model for a broad range of biological
phenomena like epidemic spread and cell division
\citep{Perthamebook:58110, Gyllenberg1990, refId0,
  iannelli1995mathematical, nla.cat-vn2330051, nla.cat-vn4515943,
  thieme2003mathematics}. The second model has many similarities with
a class of partial differential equations called
\textit{growth-fragmentation equations}. The nonlinear version we
study here was introduced in \cite{PPS14}, but on this general type of
equations we also mention the works in \citet{AFlin8doumicjauffret:hal-00408088,
  AFlin14doi:10.1142/S0218202506001480, AFlin19PERTHAME2005155,
  AFnonlin6doi:10.1080/17513750902935208, AFnonlin9ENGLER98,
  AFnonlin10FARKAS2007119, AFnonlin112011arXiv1102.2871G,
  AFnonlin13Laurenot2007, AFnonlin20SIMONETT2006580}.

The dynamics of an age-structured, interacting neuron population is
given by the integro-differential PDE
\begin{align}\label{eq:nofatigue}
  \begin{cases}
    \begin{split}
      &\frac{\p}{\p t} n(t,s) + \frac{\p}{\p s} n(t,s) + p(N(t),s)n(t,s) =0, \quad t,s > 0, \\
      &N(t):= n(t, s=0) =  \int_{0}^{+\infty} p(N(t),s)n(t,s)ds, \quad t>0,\\
      &n(t=0, s) = n_0(s), \quad s \geq 0.
    \end{split}
  \end{cases}
\end{align}
which models the evolution of a neuron population density $n(t,s)$
depending on time $t$ and the time $s$ elapsed since the last
discharge. Neurons randomly fire at a rate $p$ per unit of time, and
they re-enter the cycle from $s=0$ immediately after they fire, as
imposed through the boundary condition at $s=0$; the variable $s$ can
thus be regarded as the `age' of neurons, making a parallel with
models for birth and death processes. The global activity $N(t)$
denotes the density of neurons which are undergoing a discharge at
time $t$. If the firing rate $p$ increases with $N$, interactions are
\emph{excitatory}: the firing of neurons makes it more likely that
connected neurons will also fire; if $p$ decreases with $N$ then
interactions are \emph{inhibitory}. This nonlinear model preserves
positivity (see Section \ref{sec:nofatigue} for a rigorous proof) and
has the conservation property
\begin{equation}
  \label{cons}
  \frac{d}{dt} \int_{0}^{+\infty} n(t,s) \d s = 0.
\end{equation}
In particular, this ensures that if the density of neurons is a
probability distribution initially, then it remains so. Whenever it is
convenient we assume that $n_0$ is a probability distribution (which
may be assumed after a suitable scaling).

The second model we consider is a modified version of the first one
where $s$ represents a generic ``state'' of the neuron, not
necessarily the time elapsed since the last discharge. This model was
proposed in \cite{PPS14}, and assumes that neurons in a state $u$
return to a certain state $s < u$ after firing, with a certain
probability distribution $\kappa(s,u)$. The model reads as follows:
\begin{align} \label{eq:fatnonlin}
  \begin{cases}
    \begin{split}
      \frac{\p}{\p t} n(t,s) + \frac{\p}{\p s} n(t,s)
      &+ p(N(t),s)n(t,s)
      \\
      = \int_{0}^{+\infty} &\kappa(s,u)p(N(t),u)n(t,u)du,
      \quad u,s,t > 0,
      \\
      n(t, s=0) &= 0,
      \quad N(t):= \int_{0}^{+\infty} p(N(t),s)n(t,s)ds, \quad t>0,
      \\ 
      n(t=0,s) &= n_0(s), \quad s \geq 0.
    \end{split}
  \end{cases}
\end{align}
This equation differs from the first one in the addition of a kernel
$\kappa = \kappa(s,u)$. For fixed $u$, the quantity $\kappa(\cdot, u)$
is a probability measure which gives the distribution of neurons which
take the state $s$ when they discharge at a state $u$. Hence, neurons
do not necessarily start the cycle from $s=0$ after firing, and so $s$
cannot be properly understood as an `age' variable in this model. One
should notice that equation \eqref{eq:nofatigue} is a limiting case of
equation \eqref{eq:fatnonlin} when $\kappa(\cdot, u) = \delta_0(s)$,
the Dirac delta at $s=0$. This connection is seen more obviously when
a definition for weak or measure solutions is given; see Section
\ref{sec:fatigue}. We remark that the terms involving $p$ and
$\kappa$ are mathematically close to the ones appearing in
fragmentation processes.

These equations and similar models have been shown to exhibit many
interesting phenomena which are consistent with the experimental
behaviour of neurons: depending on the parameter $p$ and the initial
data one can find periodic solutions, apparently chaotic solutions,
and solutions which approach an equilibrium state. The first two kinds
of behaviour (periodic and chaotic solutions) are harder to study
mathematically; numerical simulations have been performed in
\cite{PPS10,PPS13} and some explicit solutions have been
found. Regarding convergence to equilibrium, some regimes are studied
in these works using perturbative techniques, such as the so-called
\emph{low-connectivity} and \emph{high-connectivity} cases. Our
contribution in this work is a simplified study of the
low-connectivity case (corresponding to a weak nonlinearity) which
gives improved results, and which uses a promising technique for this
and similar models.

As a population balance equation of a form that appears often in
mathematical biology, several techniques exist to study equations
\eqref{eq:nofatigue} and \eqref{eq:fatnonlin} rigorously. We refer to
\citet{Perthamebook:58110} for a good exposition of many of the
relevant tools. One of the main methods used so far in the study of
convergence to equilibrium for equations \eqref{eq:nofatigue} and
\eqref{eq:fatnonlin} is the \emph{entropy method}, which roughly
consists in finding a suitable Lyapunov functional $H = H(n)$ such
that
\begin{equation}
  \label{eq:Htheorem}
  \ddt H(n(t, \cdot)) = - D(n(t, \cdot)) \leq 0
\end{equation}
along solutions $n = n(t,s)$ to \eqref{eq:nofatigue} or
\eqref{eq:fatnonlin}, and then investigating whether one may prove
inequalities of the type $\lambda H(n) \leq D(n)$ for some
$\lambda > 0$ and a family of functions $n$ sufficiently large to
contain $n(t,\cdot)$ for all times $t$. If the answer is positive, one
can apply the Gronwall inequality to \eqref{eq:Htheorem} and deduce
that $H(n(t,\cdot))$ decays exponentially with a rate proportional to
$e^{-\lambda t}$. This in turn may give useful information on the
approach to equilibrium, often implying that $n(t,\cdot)$ approaches
equilibrium in the $L^1$ norm. A fundamental difficulty is that
phenomenological equations motivated by biological considerations do
not have any obvious Lyapunov functionals. This difficulty leads us to
considering cases which are close to a linear regime, taking advantage
of the fact that mass- and positivity-conserving linear equations
(essentially Markov evolutions) have a well-known family of Lyapunov
functionals. This idea was followed in \citet{PPS10,PPS13,PPS14},
using a specific Lyapunov functional obtained by integrating the
primitive of $n - n_*$ (where $n_*$ is an equilibrium state) against a
suitable weight.  

Apart from the entropy method, for the time elapsed neuron network model \eqref{eq:nofatigue}, another approach has been developed in \citet{Mischler2018, 2015arXiv151207112W}. This approach is based on spectral analysis theory for semigroups in Banach spaces. In \cite{Mischler2018}, uniqueness of the steady state and its nonlinear exponential stability in the weak connectivity regime for the first model was proved. This approach is extended in \cite{2015arXiv151207112W} to the cases without delay and with delay both in the weak and strong connectivity regimes considering a particular step function as a firing rate. Furthermore, in \cite{doi:10.1142/S021820251550058X} the link between several point processes models (Poisson, Wold, Hawkes) that have been proved to statistically fit real spike trains data and age-structured partial differential equations as introduced by \cite{PPS10} was investigated. This approach is extended to generalized Hawkes processes as microscopic models of individual neurons in \cite{CHEVALLIER20173870}.

We propose an alternative approach that is based on neither the entropy
method nor the forementioned approaches, but instead takes advantage of a set of results in the theory
of Markov processes known as \emph{Doeblin's theory}, with some
extensions such as \emph{Harris' theorem}; see \citet{Harris}, or
\citet{Hairer, Gabriel} for simplified recent proofs and
\citet[Chapter 2]{Stroock2005} for a basic exposition. The idea is
still based on first studying the linear case and then carrying out a
perturbation argument; the difference is that we study the spectral
properties of the linear operator by Doeblin's theory, which is quite
flexible and later simplifies the proofs. We obtain a spectral gap
property of the linear equation in a set of measures, and this leads
to a perturbation argument which naturally takes care of the boundary
conditions in \eqref{eq:nofatigue}-\eqref{eq:fatnonlin}. Similar ideas
are reviewed in \cite{Gabriel} for the renewal equation, and have been
recently used in \cite{BCG2017} for neuron population models
structured by voltage.

Due to this strategy, studying solutions to \eqref{eq:nofatigue} and
\eqref{eq:fatnonlin} in the sense of measures comes as a natural
setting for two important reasons: first, it fits well with the linear
theory; and second, it allows us to treat the weakly nonlinear case as
a perturbation of the linear one. Note that one difference between the
weakly nonlinear case and the linear case for equation
\eqref{eq:nofatigue} is in the boundary condition, and this is
conveniently encoded as a difference in a measure source term; see the
proof of Theorem \ref{thm:main} for details on this. Measure solutions
are also natural since a Delta function represents an initial
population whose age (or structuring variable) is known
precisely. There exist also recent works on numerical schemes for
structured population models in the space of nonnegative measures
\citep{num1, num2, num3}. Entropy methods have also been extended to
measure initial data by \cite{2016arXiv160407657G} for the renewal
equation.

\medskip

Let us detail our notation and assumptions before stating the main
results. Regarding the regularity of the firing rate $p = p(N,s)$, we
always assume that $p$ is bounded, Lipschitz and nonnegative:
\begin{equation}
  \label{p1}
  p \in W^{1,\infty}([0,+\infty) \times [0,+\infty)),
  \qquad p(N,s) \geq 0 \quad \text{ for all $N,s \in [0,+\infty)$.}
\end{equation}
and we call $L$ the Lipschitz constant of $p$ with respect to $N$;
that is, $L$ is the smallest number such that
\begin{equation}
  \label{p2}
  |p(N_1, s) - p(N_2, s)| \leq L |N_1 - N_2|
  \quad \text{for all $N_1, N_2, s \geq 0$}.
\end{equation}
We assume that for a fixed global activity $N$ the firing rate
increases as time passes; more precisely,
\begin{equation}
  \label{p3}\frac{\partial }{\partial s} p(N,s) > 0, \quad \text{for all $N, s \geq 0$}.
\end{equation} where the derivative is defined. We also assume
certain bounds on $p$: there exist $s_*, p_{\min}, p_{\max} > 0$ such
that
\begin{equation}
  \label{p4}
  p_{\min} \1_{[s_*,\infty)} \leq  p(N,s) \leq p_{\max} 
  \quad \text{for all $N, s \geq 0$,}
\end{equation}
where $\1_A$ denotes the characteristic function of a set
$A$. Regarding $\kappa$, the basic modelling assumption is that
\begin{equation}
  \label{k5}
  \text{for each $u \geq 0$, $\kappa(\cdot, u)$ is a probability
    measure supported on $[0,u]$.}
\end{equation}
We also impose a positivity condition on $\kappa$ which essentially
states that after firing there is always a sizeable probability of
jumping to a state with $s$ close to $0$: there exist
$\epsilon, 0 < \delta < s_*$ such that
\begin{equation}
  \label{k6}
  \kappa(\cdot, u) \geq \epsilon \mathds{1}_{[0,\delta]}
  \quad \text{for all $u \geq s_*$}.
\end{equation}

\bigskip Our main result can be formulated as follows for both
equation \eqref{eq:nofatigue} and \eqref{eq:fatnonlin}:

\begin{thm}
  \label{thm:main}
  We assume that \eqref{p1}--\eqref{p4} are satisfied for equation
  \eqref{eq:nofatigue}, or \eqref{p1}--\eqref{k6} for equation
  \eqref{eq:fatnonlin}. We assume also that $L$ is small enough
  depending on $p$ and $\kappa$ (with an explicit estimate; see
  remarks after the statement).

  Let $n_0$ be a probability measure on $[0,+\infty)$. There is a
  unique probability measure $n_*$ which is a stationary solution to
  \eqref{eq:nofatigue} or \eqref{eq:fatnonlin}, and there exist
  constants $C \geq 1$, $\lambda > 0$ depending only on $p$ and
  $\kappa$ such that the (mild or weak) measure solution $n = n(t)$ to
  \eqref{eq:nofatigue}-\eqref{eq:fatnonlin} satisfies
  \begin{equation}  \label{eq:mainfor2}
    \|n(t) - n_*\|_{\mathrm{TV}}
    \leq C e^{-\lambda t}
    \| n_0-n_*\|_{\mathrm{TV}}, \text{ for all } t \geq 0.
  \end{equation}
\end{thm}
\begin{rem}
  The constants are all constructive. To be precise, one can take
  \begin{gather*}
    \lambda = \lambda_1 - \tilde{C}, \qquad C = C_1 \qquad \text{for
      \eqref{eq:nofatigue}},
    \\
    \lambda = \lambda_2 - \tilde{C}, \qquad C = C_2 \qquad \text{for
      \eqref{eq:fatnonlin}},
  \end{gather*}
  where
  \begin{align*}
    &C_1 := \frac{1}{1-s_*\beta},
      \qquad
    &&\lambda_1
       = - \frac{\log (1-s_*\beta)}{2s_*}
    \\
    &C_2 := \frac{1}{1 - \epsilon \delta (s_*-\delta) \beta},
      \qquad
    &&\lambda_2
       = - \frac{\log ( 1 - \epsilon \delta (s_*-\delta) \beta) }{2s_*}
  \end{align*}	
  and with
  \begin{equation*}
    \beta = p_{\min}  e^{-2p_{\max}s_*} \text{ and } \tilde{C} = 2p_{\max} \frac{L}{1-L}.
  \end{equation*}
  The smallness condition on $L$ can be written as
  \begin{equation*}
    L < \min \bigg \{  \frac{p_{\min}^2}{ p_{\max} ^2\left(  s_*p_{\min} (s_*p_{\min}+2) +2\right)} , \frac{\log (1-s_*\beta)}{\log (1-s_*\beta) - 4 p_{\max} s_*} \bigg \}
       \text{ for
        \eqref{eq:nofatigue}}
  \end{equation*}
  or
  \begin{equation*}
    L <  \min \bigg\{      
        \frac{ p_{\min}\epsilon \delta (s_*-\delta) \beta}{ p_{\min}\epsilon \delta (s_*-\delta) \beta+ p_{\max} e^{4p_{\max}s_*}}
      , \frac{\log (1-\epsilon \delta (s_*-\delta))}
      {\log (1- \epsilon \delta (s_*-\delta)) - 4 p_{\max} s_*}
    \bigg\}
    \text{ for \eqref{eq:fatnonlin}}. 
\end{equation*}

\end{rem}

\medskip As remarked above, the closest results in the literature are
those of \citet{PPS10, PPS14}. Our equation \eqref{eq:nofatigue} is
essentially the model in \citet{PPS10}, written in a slightly
different formulation that does not include time delay and does not
highlight the connectivity as a separate parameter (the connectivity
of neurons in our case is measured in the size of $\partial_N p$). The
results in \citet{PPS10} use entropy methods and show exponential
convergence to equilibrium (a similar statement to Theorem
\ref{thm:main}) in a weighted $L^1$ space, for the case with delay and
for a particular form of the firing rate $p$. As compared to this, our
results work in a space of measures and can be easily written for
general firing rates $p$; however, we have not considered the
large-connectivity case (which would correspond to large
$\partial_N p$ in our case) or the effects of time delay.

Similar remarks apply to the results for equation \eqref{eq:fatnonlin}
contained in \citet{PPS14}. In this case our strategy gives in general
conditions which are simpler to state, and provide a general framework
which may be applied to similar models. Again, we have not considered
a time delay in the equation, which is a difference with the above
work. There are numerical simulations and further results on regimes
with a stronger nonlinearity in \citet{PPS10, PPS13, PPS14}.

\medskip

This paper is organized as follows. In Section \ref{sec:sg} we state
Doeblin's Theorem from the theory of Markov processes, which
plays a crucial role in our convergence results. Section
\ref{sec:nofatigue} and section \ref{sec:fatigue} are dedicated to the
age-structured neuron population model \eqref{eq:nofatigue} and the
structured neuron population model with a fatigue \eqref{eq:fatnonlin}
respectively. The organisation of those last two sections is the same: in each,
we first give proofs for well-posedness, and existence and uniqueness
of stationary solutions, always in the weak nonlinearity regime. Later
we consider the linear equations (when $p$ does not depend on $N$) and
we prove that solutions have positive lower bounds which ensures that
the associated stochastic semigroups satisfy the Doeblin condition,
obtaining exponential convergence results for the linear
problems. Finally, we prove exponential relaxation to the steady state
for the nonlinear models \eqref{eq:nofatigue}--\eqref{eq:fatnonlin} by
a perturbation argument based on the linear theory.

\section{Spectral gaps in total variation}
\label{sec:sg}

We briefly present here the version of Doeblin's theorem that we use
in this paper. A recent proof which brought it to our attention can be
found in \cite{Gabriel}; a general version known as Harris' theorem
can be found in \citet{Harris, Hairer}. This result applies to
stochastic semigroups defined in a space of measures (or in an $L^1$
space); that is, mass- and positivity-preserving
semigroups\footnote{In the literature a Markov semigroup is often the
  dual of this kind of semigroup, but is sometimes also the same as
  our definition of a stochastic semigroup. We give the definition to
  avoid any confusion.}.

Given a measurable space $(E, \mathcal{A})$ we denote by
$\mathcal{M}(E)$ the set of finite measures on $E$, and by
$\mathcal{P}(E)$ the set of probability measures on $E$. We equip
$\mathcal{M}(E)$ with the usual total variation norm which we denote
by $\|\cdot\|_{\mathrm{TV}}$; we recall that this norm is defined by
$\|\mu\| := \int_E \mu_+ + \int_E \mu_-$, where $\mu = \mu_+ - \mu_-$
is the Hahn-Jordan decomposition of the measure $\mu$ into its
positive and negative parts.

\begin{dfn}
  Let $(E, \mathcal{A})$ be a measurable space. A \emph{stochastic
    operator} on $\mathcal{M}(E)$ is a linear operator
  $S \: \mathcal{M}(E) \to \mathcal{M}(E)$ such that $S \mu \geq 0$
  for all $\mu \geq 0$, and such that $\int_E S \mu = \int_E \mu$ for
  all $\mu \in \mathcal{M}(E)$. Equivalently, it is a linear operator
  which preserves the space $\mathcal{P}(E)$ of probability measures
  on $E$.

  A \emph{stochastic semigroup} on $\mathcal{M}(E)$ is a semigroup
  $(S_t)_{t \geq 0}$ of operators
  $S_t \: \mathcal{M}(E) \to \mathcal{M}(E)$ such that $S_t$ is a
  stochastic operator for each $t \geq 0$. A \emph{stationary state}
  of a semigroup $(S_t)_{t \geq 0}$ on $\mathcal{M}(E)$ is a measure
  $\mu \in \mathcal{M}(E)$ such that $S_t \mu = \mu$ for all
  $t \geq 0$.
\end{dfn}

We observe that $\|S_t \mu\|_{\mathrm{TV}} \leq \|\mu\|_{\mathrm{TV}}$
holds true for all $\mu \in \mathcal{M}(E)$. Since $S_t$ preserves
order, we have $-S_t|\mu| \leq S_t \mu \leq S_t |\mu|$ which implies
$|S_t \mu| \leq S_t |\mu|$. Then we integrate both sides to obtain
$\int |S_t \mu| \leq \int S_t \mu =\int |\mu|$ for all
$\mu \in \mathcal{M}(E)$ by using the mass conservation property of
$S_t$.

\begin{dfn}
  We say a stochastic operator $S$ satisfies the \emph{Doeblin
    condition} when there exist $0 < \alpha < 1$ and a probability
  measure $\nu$ on $(E, \mathcal{A})$ such that
  \begin{equation}
    \label{eq:Doeblin}
    S \mu \geq \alpha \nu
    \qquad \text{for all $\mu \in \mathcal{P}(E)$.}
  \end{equation}
\end{dfn}

The following version of Doeblin's theorem is essentially the same as
the one in \citet{Gabriel} or \citet[Chapter 2]{Stroock2005}, and a
particular case of \citet{Hairer} if one ignores some technical
conditions regarding the existence of a kernel for the semigroup. We
give a short proof here for clarity, essentially along the lines of
\citet{Gabriel}.

\begin{thm}[Semigroup version of Doeblin's theorem]
  \label{thm:Doeblin-semigroup}
  Let $(E, \mathcal{A})$ be a measurable space and $(S_t)_{t \geq 0}$
  a stochastic semigroup on $\mathcal{M}(E)$. If there exists
  $t_0 > 0$ such that $S_{t_0}$ satisfies the Doeblin condition
  \eqref{eq:Doeblin} then the semigroup has a unique equilibrium $n_*$
  in $\mathcal{P}(E)$, and
  \begin{equation}
    \label{eq:Doeblin-semigroup-decay}
    \| S_t (n - n_*) \|_{\mathrm{TV}} \leq
    \frac{1}{1-\alpha} e^{- \lambda t} \| n-n_* \|_{\mathrm{TV}},
    \text{ for all } t \geq 0,
  \end{equation}
  for all $n \in \mathcal{P}(E)$, where
  \begin{equation*}
    \lambda := -\frac{\log (1-\alpha)}{t_0} > 0.
  \end{equation*}
  In addition,
  \begin{equation}
    \label{eq:St0-contractive-Doeblin}
    \| S_{t_0} (n_1 - n_2) \|_{\mathrm{TV}}
    \leq
    (1-\alpha) \|n_1 - n_2\|_{\mathrm{TV}}
  \end{equation}
  for any probability measures $n_1, n_2$ on $E$.
\end{thm}

\begin{proof}
  By the triangle inequality we have
  \begin{equation*}
    \| S_{t_0} \mu_1 - S_{t_0} \mu_2 \|_{\mathrm{TV}}
    \leq
    \| S_{t_0} \mu_1 - \alpha \nu \|_{\mathrm{TV}}
    + \| S_{t_0} \mu_2  - \alpha \nu \|_{\mathrm{TV}}.
  \end{equation*}
  Now, since $S_{t_0} \mu_1 \geq \alpha \nu$, we can write
  \begin{equation*}
    \| S_{t_0} \mu_1 - \alpha \nu \|_{\mathrm{TV}}
    =
    \int (S_{t_0} \mu_1 - \alpha \nu)
    =
    \int \mu_1 - \alpha
    = 1-\alpha,
  \end{equation*}
  due to mass conservation, and similarly for the term
  $\| S_{t_0} \mu_2 - \alpha \nu \|_{\mathrm{TV}}$. This gives
  \begin{equation}
    \label{eq:contrac-Doeblin}
    \| S_{t_0} \mu_1 - S_{t_0} \mu_2 \|_{\mathrm{TV}}
    \leq
    2(1-\alpha)
    = (1-\alpha) \|\mu_1 - \mu_2\|_{\mathrm{TV}}.
  \end{equation}
  if $\mu_1, \mu_2 \in \mathcal{P}(E)$ have disjoint support.  By
  homogeneity, this inequality is obviously also true for any
  nonnegative $\mu_1, \mu_2 \in \mathcal{M}(E)$ having disjoint
  support with $\int \mu_1 = \int \mu_2$. We obtain the inequality in
  general for any $\mu_1, \mu_2 \in \mathcal{M}(E)$ with the same
  integral by writing
  $\mu_1 - \mu_2 = (\mu_1 - \mu_2)_+ - (\mu_2 - \mu_1)_+$, which is a
  difference of nonnegative measures with the same integral. This
  shows \eqref{eq:St0-contractive-Doeblin}.

  The contractivity \eqref{eq:contrac-Doeblin} shows that the operator
  $S_{t_0}$ has a unique fixed point in $\mathcal{P}(E)$, which we
  call $n_*$. In fact, $n_*$ is a stationary state of the whole
  semigroup since for all $s \geq 0$ we have
  \begin{equation*}
     S_{t_0} S_{s} n_* = S_{s} S_{t_0} n_* = S_{s} n_*,
   \end{equation*}
   which shows that $S_{s} n_*$ (which is again a probability measure)
   is also a stationary state of $S_{t_0}$; due to uniqueness,
   \begin{equation*}
     S_s n_* = n_*.
   \end{equation*}
   Hence the only stationary state of $(S_t)_{t \geq 0}$ must be
   $n_*$, since any stationary state of $(S_t)_{t \geq 0}$ is in
   particular a stationary state of $S_{t_0}$.

   In order to show \eqref{eq:Doeblin-semigroup-decay}, for any
   $n \in \mathcal{P}(E)$ and any $t \geq 0$ we write
   \begin{equation*}
     k := \lfloor{t/t_0}\rfloor,
   \end{equation*}
   (where $\lfloor \cdot \rfloor$ denotes the integer part) so that
   \begin{equation*}
     \frac{t}{t_0} - 1 < k \leq \frac{t}{t_0}.
   \end{equation*}
   Then,
   \begin{multline*}
     \| S_t (n - n_*) \|_{\mathrm{TV}}
     = \| S_{t - k t_0} S_{k
       t_0} (n-n_*) \|_{\mathrm{TV}}
     \leq
     \| S_{k t_0} (n-n_*) \|_{\mathrm{TV}}
     \\
     \leq
     (1-\alpha)^{k} \| n-n_* \|_{\mathrm{TV}}
     \leq
     \frac{1}{1-\alpha}
     \exp\left( 
       \frac{t \log (1-\alpha)}{t_0} \right) \| n-n_* \|_{\mathrm{TV}}.
     \qedhere
   \end{multline*}
\end{proof}

\section{An age-structured neuron population model}

\label{sec:nofatigue}

In this section we consider equation \eqref{eq:nofatigue} for an age-structured neuron population. We first
develop a well-posedness theory in the sense of measures, and then we
use Doeblin's Theorem \ref{thm:Doeblin-semigroup} for the linear
problem \eqref{eq:nofatigue-linear} to show exponential convergence to the equilibrium. After
giving conditions for existence and uniqueness of a stationary
solution to equation \eqref{eq:nofatigue}, we use a perturbation
argument in order to obtain a result on its asymptotic behaviour.

\subsection{Well posedness} \label{sec:well-posedness}

In order to develop our well-posedness theory in measures we need to
introduce our notation and the norms we will be considering. We denote
$\R^+_0 := [0,+\infty)$, and $\mathcal{M}(\R^+_0)$ is the set of
finite, signed Borel measures on $\R^+_0$. $\mathcal{M}_+(\R^+_0)$
denoted the subset of $\mathcal{M}(\R^+_0)$ formed by the nonnegative
measures. Since we will always work in $\R^+_0$, for simplicity we
will often write $\mathcal{M}$ and $\mathcal{M}_+$ to denote these
sets, respectively.

We often identify a measure $\mu \in \mathcal{M}(\R^+_0)$ with its
density with respect to Lebesgue measure, denoting the latter by the
function $\mu = \mu(s)$. We abuse notation by writing $\mu(s)$ even
for measures that may not have a density with respect to Lebesgue
measure. Similarly, for a function
$n \: [0, T) \to \mathcal{M}(\mathbb{R}^+_0)$ we may often write
$n(t, s)$ even if $n(t)$ does not have a density with respect to
Lebesgue measure. In these cases any identities involved should be
understood as identities between measures.

We denote by $\mathcal{C}_0(\R^+_0) \equiv \mathcal{C}_0$ the set of
continuous functions $\phi$ on $\R^+_0$ with
$\lim_{s \to +\infty} \phi(s) = 0$ endowed with the supremum norm
\begin{equation*}
\| \phi \|_{\infty} := \sup_{s \geq 0} |\phi(s)|,
\end{equation*}
$\mathcal{C}_0(\R^+_0)$ becomes a Banach space. Similarly,
$\mathcal{C}_{\mathrm{c}}(\R^+_0) \equiv \mathcal{C}_{\mathrm{c}}$
denotes the set of compactly supported continuous functions on
$[0,+\infty)$.

In $\mathcal{M}$ one can define the usual total variation norm, which
we will denote by $\|\cdot\|_{\mathrm{TV}}$. We recall that
$(\mathcal{M}, \|\cdot\|_{\mathrm{TV}})$ is a Banach space, and is the
topological dual of $\mathcal{C}_0([0,+\infty))$ with the supremum
norm, as stated by the Riesz representation theorem. The weak-$*$
topology on $\mathcal{M}$ is the weakest topology that makes all
functionals $T \: \mathcal{M} \to \R$, $\mu \mapsto \int_{\R^+_0}
\phi \mu$ continuous, for all $\phi \in \mathcal{C}_0$. In the
associated topology, a sequence $(\mu_k)_{k \geq 1}$ in $\mathcal{M}$ converges in the
weak-$*$ sense to $\mu \in \mathcal{M}$ when
\begin{equation*}
\lim_{k \to +\infty} \int_{\R^+_0} \phi \mu_k
= \int_{\R^+_0} \phi \mu
\qquad \text{for all $\phi \in \mathcal{C}_0$}.
\end{equation*}

We will also use the bounded Lipschitz norm $\|\cdot\|_{\mathrm{BL}}$
on $\mathcal{M}$, sometimes known as the flat metric or the
$W^{1,\infty}$ dual metric, defined by
\begin{equation*}
\| \mu \|_{\mathrm{BL}} := \sup_{\psi \in \mathcal{L}} \int_{\R^+_0}
\psi \mu,
\qquad \mu \in \mathcal{M}
\end{equation*}
where
\begin{equation*}
\mathcal{L} := \{ \psi \in \mathcal{C}(\R^+_0) \mid
\text{$\psi$ bounded and Lipschitz with $\|\psi\|_\infty +
	\|\psi'\|_\infty \leq 1$} \}.
\end{equation*}
One sees from this definition that the bounded Lipschitz norm is dual
to the norm
\begin{equation*}
\| \psi \|_{1,\infty} := \|\psi\|_\infty + \|\psi'\|_\infty,
\qquad \psi \in W^{1,\infty}(\R^+_0)
\end{equation*}
defined on
$W^{1,\infty}(\R^+_0) = \{ \psi \in \mathcal{C}(\R^+_0) \mid
\text{$\psi$ bounded and Lipschitz} \}$ (but
$(\mathcal{M}, \|\cdot\|_{\mathrm{BL}})$ is \emph{not} the topological
dual of $W^{1,\infty}$). An important property of this norm is that it
metrises the weak-$*$ topology on any tight set with bounded total
variation. We recall that a set $B \subseteq \mathcal{M}$ is
\emph{tight} if for every $\epsilon > 0$ there exists $R > 0$ such
that $|\mu|((R,+\infty)) < \epsilon$ for all $\mu \in B$.

\begin{lem}[\citet{Lorenz}, 2.5.1, Proposition 43]
  \label{lem:BL-weak}
  If $B \subseteq \mathcal{M}$ is tight and is bounded in total
  variation norm, then the topology associated to
  $\|\cdot\|_{\mathrm{BL}}$ on $B$ is equal to the weak-$*$ topology
  on $B$.
\end{lem}

If $I \subseteq \R$ is an interval we denote by
$\mathcal{C}(I, \mathcal{M}_+(\R_0^+))$ the set of functions
$n \: I \to \mathcal{M}_+(\R_0^+)$ which are continuous with respect
to the bounded Lipschitz norm on $\mathcal{M}_+(\R_0^+)$.
We define mild measure solutions to equation \eqref{eq:nofatigue} by
the usual procedure of rewriting it using Duhamel's formula. We denote
by $(T_t)_{t \geq 0}$ the translation semigroup generated on
$(\mathcal{M}, \|\cdot\|_{\mathrm{BL}})$ by the operator
$-\partial_s$. That is: for $t \geq 0$, any measure
$n \in \mathcal{M}(\R^+_0)$ and any $\phi \in \mathcal{C}_0(\R^+_0)$,
\begin{equation}
\label{eq:Tt-def}
\int_{\R^+_0} \phi(s) T_t n(s) \d s
:= \int_{\R^+_0} \phi(s+t) n(s) \d s.
\end{equation}
In other words, using the notation we follow in this paper,
\begin{equation*}
T_t n(s) := n(s-t),
\end{equation*}
with the understanding that $n$ is zero on $(-\infty, 0)$.

\begin{dfn}
	\label{defnmeasuresoln}
	Assume $p$ satisfies \eqref{p1} and is nonnegative. A couple of
	functions $n \in \mathcal{C}([0, T), \mathcal{M}_+(\R^+_0))$ and
	$N \in \mathcal{C}([0,T), [0,+\infty))$, defined on an interval
	$[0, T)$ for some $T \in (0,+\infty)$, is called a \emph{mild
		measure solution} to \eqref{eq:nofatigue} with initial data
	$n_0 \in \mathcal{M}(\R^+_0)$ and $N_0 \in \R$ if it satisfies
	$n(0) = n_0$, $N(0) = N_0$,
	\begin{equation}
	\label{soln1}
	n(t,s) = T_t n_0(s)
	- \int_0^t T_{t-\tau} \big(p (N(\tau), \cdot)
	n(\tau,\cdot) \big) (s) \d \tau
	+ \int_0^t T_{t-\tau} \big( N(\tau) \delta_0 \big) (s) \d \tau
	\end{equation}
	for all $t \in [0,T)$, and
	\begin{equation*}
	N(t) = \int_0^\infty p(N(t), s) n(t,s) \d s,
	\qquad t \in [0, T).
	\end{equation*}
\end{dfn}

\begin{rem}
	We notice that the second term in \eqref{soln1} can be rewritten as
	\begin{multline}
	\label{eq:delta-integral-alternative}
	\int_0^t T_{t-\tau} \big( N(\tau) \delta_0 \big) (s) \d \tau
	=
	\int_0^t N(\tau) \delta_{t-\tau} (s) \d \tau
	\\
	= N(t-s) \1_{[0,t]}(s)
	= N(t-s) \1_{[0,\infty)}(t-s).
	\end{multline}
	This will sometimes be a more convenient form.
\end{rem}

By integrating in $\R^+_0$, Definition \ref{defnmeasuresoln} directly
implies mass conservation:
\begin{lem}[Mass conservation for measure solutions]
	\label{lem:mass-conservation}
	Let $T \in (0,+\infty]$. Any mild measure solution $(n, N)$ to
	\eqref{eq:nofatigue} defined on $[0,T)$ satisfies
	\begin{equation}
	\label{eq:mass-conservation1}
	\int_{\R^+_0} n(t,s) \d s = \int_{\R^+_0} n_0(s) \d s,
	\qquad \text{for all $t \in [0,T)$,}
	\end{equation}
	or in other words (since solutions are nonnegative measures by
	definition),
	\begin{equation*}
	\| n(t) \|_{\mathrm{TV}} = \|n_0\|_{\mathrm{TV}}
	\qquad \text{for all $t \in [0,T)$.}
	\end{equation*}
\end{lem}

\begin{lem}
  \label{lem:N-stability}
  Assume that $p$ satisfies \eqref{p1} and the Lipschitz constant $L$
  in \eqref{p2} satisfies $L < 1 / \|n\|_{\mathrm{TV}}$, and let
  $n \in \mathcal{M}(\R^+_0)$. There exists a unique $N \in \R$
  satisfying
  \begin{equation}
    \label{eq:N-def}
    N = \int_0^\infty p(N, s) n(s) \d s.
  \end{equation}
  Under these conditions, if $n_1, n_2 \in \mathcal{M}(\R^+_0)$ are
  two measures and $N_1, N_2 \in \R$ are the corresponding solutions
  to \eqref{eq:N-def}, then
  \begin{equation}
    \label{eq:N-stability}
    |N_1 - N_2|
    \leq
    \frac{\|p\|_{\infty}}{1 - L \|n_1\|_{\mathrm{TV}}}
    \|n_1 - n_2\|_{\mathrm{TV}}.
  \end{equation}
\end{lem}

\begin{proof}
  We define the map $\Phi \: \R \to \R$ by
  \begin{equation*}
    \Phi(N) := \int_0^\infty p(N, s) n(s) \d s,
  \end{equation*}
  and we notice that for any $N_1, N_2 \in \R$,
  \begin{equation*}
    |\Phi(N_1) - \Phi(N_2)|
    \leq
    \|p(N_1, \cdot) - p(N_2, \cdot)\|_{\infty}
    \|n\|_{\mathrm{TV}}
    \leq
    L |N_1 - N_2| \|n\|_{\mathrm{TV}}.
  \end{equation*}
  Since $L < 1 / \|n\|_{\mathrm{TV}}$, the map $\Phi$ is contractive
  and has a unique fixed point, which is a solution to
  \eqref{eq:N-def}. For the second part of the lemma, consider
  $n_1, n_2 \in \mathcal{M}(\R^+_0)$ and $N_1$, $N_2$ the corresponding
  solutions to \eqref{eq:N-def}. Then
  \begin{multline*}
    |N_1 - N_2|
    \leq
    \int_0^\infty |p(N_1, s) - p(N_2,s)| n_1(s) \d s
    +
    \left|
      \int_0^\infty p(N_2, s) (n_1(s) - n_2(s)) \d s
    \right|
    \\
    \leq
    L |N_1 - N_2| \|n_1\|_{\mathrm{TV}}
    +
    \|p\|_{\infty} \|n_1 - n_2 \|_{\mathrm{TV}},
  \end{multline*}
  which shows \eqref{eq:N-stability}.
\end{proof}

\begin{thm}[Well-posedness of \eqref{eq:nofatigue} in measures]
  \label{thmwellposed}
  Assume that $p$ satisfies \eqref{p1} and the Lipschitz constant $L$
  in \eqref{p2} satisfies $L \leq 1/(4 \|n_0\|_{\mathrm{TV}})$.  For
  any given initial data $n_0 \in \mathcal{M}_+(\mathbb{R}_0^+)$ there
  exists a unique measure solution
  $n \in \mathcal{C}([0,+\infty); \mathcal{M}_+(\R^+_0))$ of
  \eqref{eq:nofatigue} in the sense of Definition
  \ref{defnmeasuresoln}. In addition, if $n_1$, $n_2$ are any two mild
  measure solutions to \eqref{eq:nofatigue} (with possibly different
  initial data) defined on any interval $[0,T)$ then
  \begin{equation}
    \label{eq:t-stability1}
    \| n_1(t) - n_2(t) \|_{\mathrm{TV}}
    \leq
    \| n_1(0) - n_2(0) \|_{\mathrm{TV}}
    \,
    e^{4 \|p\|_{\infty} t }
    \qquad
    \text{for all $t \in [0,T)$.}
  \end{equation}
\end{thm}

\begin{rem}
  We notice that the condition that $L$ is small is already needed
  here, since otherwise the problem is \emph{not} well-posed: consider
  for example the case $p(N,s) := N$, for which a solution should
  satisfy
  \begin{equation*}
    N(t) = N(t) \int_0^\infty n(t,s) \d s,
  \end{equation*}
  which only allows two options: either $N(t) = 0$ or
  $\int_0^\infty n(t,s) \d s = 1$. If $\int_0^\infty n_0(s) \d s = 1$
  then the second option holds and there are infinitely many solutions
  (since the choice of $N = N(t)$ is free). If
  $\int_0^\infty n_0(s) \d s \neq 1$ then $N(t)$ must be $0$ for all
  $t > 0$. In this latter case, either $N_0 = 0$ (and then the only
  solution is just pure transport: $n(t,s) = n_0(s-t)$ for $s > t$,
  $n(t,s) = 0$ otherwise) or $N_0 \neq 0$ (and there there are no
  solutions).

  Similar ill-posed examples can be easily designed with firing rates
  of the form $p(N,s) = f(N)g(s)$.
\end{rem}

\begin{proof}[Proof of Theorem \ref{thmwellposed}]
	The proof of this result is a standard fixed-point argument as
	followed for example in \citet{CCC13}, or in \citet{PPS10} for $L^1$
	solutions.

	Let us first show existence of a solution for a nonnegative initial
	measure $n_0 \in \mathcal{M}_+(\R^+_0)$. If $n_0 = 0$ it is clear that
	setting $n(t)$ equal to the zero measure on $\R^+_0$ for all $t$
	defines a solution, so we assume $n_0 \neq 0$. Fix $C, T > 0$, to be
	chosen later. Consider the complete metric space
	\begin{equation*}
	\mathcal{X}=
	\{ n \in \mathcal{C}([0,T], \mathcal{M}_+(\R^+_0)) \mid
	n(0) = n_0,
	\
	\|n(t) \|_{\mathrm{TV}} \leq C \text{ for all $t \in [0,T]$} \},
	\end{equation*}
	endowed with the norm
	\begin{equation*}
	\|n\|_{\mathcal{X}} := \sup_{t \in [0,T]} \|n(t)\|_{\mathrm{TV}}.
	\end{equation*}
	We remark that $\mathcal{C}([0,T], \mathcal{M}(\R^+_0))$ refers to
	functions which are continuous in the bounded Lipschitz topology,
	\emph{not} in the total variation one. Define an operator
	$\Psi \: \mathcal{X} \to \mathcal{X}$ by
	\begin{align}
	\label{bp}
	\Psi[n](t):= T_t n_0
	- \int_0^t T_{t-\tau} \big(p (N(\tau), \cdot)
	n(\tau) \big) \d \tau
	+
	\int_0^t T_{t-\tau} \big( N(\tau) \delta_0 \big) \d \tau
	\end{align}
	for all $n \in \mathcal{X}$, where $N(t)$ is defined implicitly (see
	Lemma \ref{lem:N-stability}) as
	\begin{equation}
	\label{eq:N-from-n}
	N(t) = \int_0^\infty p(N(t),s) n(t,s) \d s
	\qquad
	\text{for $t \in [0,T]$}.
	\end{equation}
	The definition of $\Psi[n]$ indeed makes sense, since both
	$\tau \mapsto T_{t-\tau} \big(p (N(\tau), \cdot) n(\tau) \big)$ and
	$\tau \mapsto T_{t-\tau} \big( N(\tau) \delta_0 \big)$ are
	continuous functions from $[0,T]$ to $\mathcal{M}(\R^+_0)$, hence
	integrable (in the sense of the Bochner integral).
	
	We first check that $\Psi[n]$ is indeed in $\mathcal{X}$. It is easy
	to see that $t \mapsto \Psi[n](t)$ is continuous in the bounded
	Lipschitz topology, and it is a nonnegative measure for each
	$t \in [0,T]$. We also have
	\begin{multline*}
	\|\Psi[n](t)\|_{\mathrm{TV}}
	\leq
	\|n_0\|_{\mathrm{TV}}
	+
	\int_0^t \| p (N(\tau), \cdot)
	n(\tau,\cdot) \|_{\mathrm{TV}} \d \tau
	+ \int_0^t \| N(\tau) \delta_0 \|_{\mathrm{TV}} \d \tau
	\\
	\leq
	\|n_0\|_{\mathrm{TV}}
	+
	T C \|p\|_\infty 
	+ T \|N\|_{L^\infty([0,T])}
	\leq
	\|n_0\|_{\mathrm{TV}}
	+
	2 T C \|p\|_\infty .
	\end{multline*}
	We choose
	\begin{equation}
	\label{eq:rt1}
	T \leq \frac{1}{4 \|p\|_\infty}
	\quad \text{and} \quad
	C := 2 \|n_0\|_{\mathrm{TV}},
	\end{equation}
	so that
	\begin{equation*}
	\|n_0\|_{\mathrm{TV}}
	+
	2 T C \|p\|_\infty
	\leq
	\|n_0\|_{\mathrm{TV}}
	+
	\frac{C}{2}
	\leq
	C.
	\end{equation*}
	Hence with these conditions on $T$ and $C$ we have
	$\Psi[n] \in \mathcal{X}$.
	
	Let us show that $\Psi$ is a contraction mapping. Take
	$n_1, n_2 \in \mathcal{X}$ and let $N_1, N_2$ be defined by
	\eqref{eq:N-from-n} corresponding to $n_1$ and $n_2$,
	respectively. We have
	\begin{multline*}
	\| \Psi[n_1](t) - \Psi[n_2](t) \|_{\mathrm{TV}}
	\leq
	\int_0^t \| (p (N_1(\tau), \cdot) - p(N_2(\tau),\cdot))
	n_1(\tau,\cdot)  \|_{\mathrm{TV}} \d \tau
	\\
	+ \int_0^t
	\| p (N_2(\tau), \cdot) (n_1(\tau) - n_2(\tau)) \|_{\mathrm{TV}} \d \tau
	+ \int_0^t \| (N_1(\tau) - N_2(\tau)) \delta_0 \|_{\mathrm{TV}} \d \tau
	\\
	=: T_1 + T_2 + T_3.
	\end{multline*}
	We bound each term separately. For $T_1$, since
	$L \leq 1 / ({4 \|n_0\|_{\mathrm{TV}}})$, using
	Lemma~\ref{lem:N-stability} we have
	\begin{multline}
	\label{eq:b1}
	T_1
	\leq
	T C L \sup_{\tau \in [0,T]} | N_1(\tau) - N_2(\tau)|
	\\
	\leq
	2 T C L \|p\|_{\infty} \| n_1 - n_2 \|_{\mathcal{X}}
	\leq
	T \|p\|_{\infty} \| n_1 - n_2 \|_{\mathcal{X}}.
	\end{multline}
	For $T_2$,
	\begin{equation}
	\label{eq:b2}
	T_2 \leq T \|p\|_{\infty} \|n_1 - n_2\|_{\mathcal{X}},
	\end{equation}
	and for $T_3$, using again Lemma \ref{lem:N-stability},
	\begin{equation}
	\label{eq:b3}
	T_3 \leq
	T \sup_{\tau \in [0,T]} | N_1(\tau) - N_2(\tau)|
	\leq
	2 T \|p\|_{\infty} \|n_1 - n_2\|_{\mathcal{X}}.
	\end{equation}
	Putting equations \eqref{eq:b1}--\eqref{eq:b3} together and taking
	the supremum over $0 \leq t \leq T$,
	\begin{equation*}
	\| \Psi[n_1] - \Psi[n_2] \|_{\mathcal{X}}
	\leq
	4 T \|p\|_{\infty} \|n_1 - n_2\|_{\mathcal{X}}.
	\end{equation*}
	Taking now $T \leq 1 / (8 \|p\|_{\infty})$ ensures that $\Psi$ is
	contractive, so it has a unique fixed point in $\mathcal{X}$, which
	is a mild measure solution on $[0,T]$. If we call $n$ this fixed
	point, since $\|n(T)\|_{\mathrm{TV}} = \|n_0\|_{\mathrm{TV}}$ by
	mass conservation (see Lemma \ref{lem:mass-conservation}), we may
	repeat this argument to continue the solution on $[T, 2T]$, $[2T,
	3T]$, showing that there is a solution defined on $[0,+\infty)$.
	
	In order to show stability of solutions with respect to the initial
	data (which implies uniqueness of solutions), take two measures
	$n_0^1, n_0^2 \in \mathcal{M}_+(\R^+_0)$, and consider two solutions
	$n_1$, $n_2$ with initial data $n_0^1$, $n_0^2$ respectively. We
	have
	\begin{multline}
	\label{eq:stab-prf-1}
	\| n_1(t) - n_2(t) \|_{\mathrm{TV}}
	\leq
	\| n_0^1 - n_0^2\|_{\mathrm{TV}}
	+
	\int_0^t \| (p (N_1(\tau), \cdot) - p(N_2(\tau),\cdot))
	n_1(\tau,\cdot)  \|_{\mathrm{TV}} \d \tau
	\\
	+ \int_0^t
	\| p (N_2(\tau), \cdot) (n_1(\tau) - n_2(\tau)) \|_{\mathrm{TV}} \d \tau
	+ \int_0^t \| (N_1(\tau) - N_2(\tau)) \delta_0 \|_{\mathrm{TV}} \d \tau,
	\end{multline}
	and with very similar arguments as before we obtain that
	\begin{multline}
	\label{eq:stab-prf-2}
	\| n_1(t) - n_2(t) \|_{\mathrm{TV}}
	\leq
	\| n_0^1 - n_0^2\|_{\mathrm{TV}}
	+ 2 L \|n_0\|_{\mathrm{TV}} \|p\|_{\infty} \int_0^t \|n_1(\tau)
	- n_2(\tau) \|_{\mathrm{TV}} \d \tau
	\\
	+ \|p\|_\infty \int_0^t 
	\| n_1(\tau) - n_2(\tau) \|_{\mathrm{TV}} \d \tau
	+ 2 \|p\|_{\infty}
	\int_0^t \| n_1(\tau) - n_2(\tau) \|_{\mathrm{TV}} \d \tau
	\\
	\leq
	\| n_0^1 - n_0^2\|_{\mathrm{TV}}
	+ 4 \|p\|_{\infty}
	\int_0^t \|n_1(\tau) - n_2(\tau) \|_{\mathrm{TV}} \d \tau.
	\end{multline}
	Gronwall's inequality then implies \eqref{eq:t-stability1}.  
\end{proof}

\paragraph{Weak solutions}

Definition \ref{defnmeasuresoln} is convenient for finding solutions,
but later we will need a more manageable form:

\begin{dfn}(Weak solution to \eqref{eq:nofatigue})
	\label{dfn:weak_solution}
	Assume $p$ satisfies \eqref{p1} and is nonnegative. A couple of
	functions $n \in \mathcal{C}([0, T), \mathcal{M}_+(\R^+_0))$ and
	$N \in \mathcal{C}([0,T), [0,+\infty))$, defined on an interval
	$[0, T)$ for some $T \in (0,+\infty]$, is called a \emph{weak
		measure solution} to \eqref{eq:nofatigue} with initial data
	$n_0 \in \mathcal{M}(\R^+_0)$ and $N_0 \in \R$ if it satisfies
	$n(0) = n_0$, $N(0) = N_0$, and for each $\varphi \in
	\mathcal{C}^\infty_{\mathrm{c}}(0,+\infty)$ the function $t \mapsto
	\int_0^\infty \varphi(s) n(t,s) \d s$ is absolutely continuous and
	\begin{multline}
	\label{eq:weak_solution}
	\ddt \int_0^\infty \varphi(s) n(t,s) \d s
	\\
	= \int_0^\infty \p_s \varphi(s) n(t,s) \d s
	- \int_0^\infty p (N(t), s) n(t,s) \varphi(s) \d s
	+ \int_0^\infty N(t) \delta_0(s) \varphi(s) \d s.
	\end{multline}
	for almost all $t \in [0,T)$, and
	\begin{equation*}
	N(t) = \int_0^\infty p(N(t), s) n(t,s) \d s,
	\qquad \text{for all $t \in [0, T)$.}
	\end{equation*}
	
\end{dfn}

Equivalence results between definitions based on the Duhamel formula
and definitions of weak solutions based on integration against a test
function are fairly common. Here we use the main theorem in
\citet{Ball1977} with
$f(t,\cdot) = -p(N(t), \cdot) n(t,\cdot) - N(t) \delta_0(\cdot)$,
which implies that mild solutions of our equation are are equivalent
to weak solutions:

\begin{thm}[\citet{Ball1977}]
	\label{thm:mild-weak}
	Assume $p$ satisfies \eqref{p1} and is nonnegative, and take
	$T \in (0,+\infty]$. A function
	$n \: [0,T) \to \mathcal{M}_+(\R^+_0)$ is a weak measure solution
	(cf. Definition \ref{dfn:weak_solution}) to \eqref{eq:nofatigue} if
	and only if it is a mild measure solution (cf. Definition
	\ref{defnmeasuresoln}).
\end{thm}

\subsection{The linear equation}
\label{sec:linear}

When $p = p(N,s)$ does not depend on $N$, equation
\eqref{eq:nofatigue} becomes linear:
\begin{align}
\label{eq:nofatigue-linear}
\begin{cases}
\begin{split}
&\frac{\p}{\p t} n(t,s)
+ \frac{\p}{\p s} n(t,s) + p(s)n(t,s) =0,
\quad t,s >0,
\\
&N(t) := n(t, s=0) =  \int_{0}^{+\infty} p(s)n(t,s) \d s,
\quad t>0,
\\
&n(t=0, s) = n_0(s), \quad s \geq 0.
\end{split}
\end{cases}
\end{align}
This is referred to as the ``no-connectivity'' case or the ``$J=0$
case'' in \citet{PPS10}.

\subsubsection{Well-posedness}

\label{sec:linear-well-posed}

We give a similar definition for \textit{mild measure solutions}:

\begin{dfn}
  \label{defnmeasuresoln-linear}
  Assume $p \: [0,+\infty) \to [0,+\infty)$ is a bounded measurable
  function. A function
  $n \in \mathcal{C}([0, T), \mathcal{M}_+(\R^+_0))$, defined on an
  interval $[0, T)$ for some $T \in (0,+\infty]$, is called a
  \emph{mild measure solution} to \eqref{eq:nofatigue-linear} with
  initial data $n_0 \in \mathcal{M}(\R^+_0)$ if it satisfies
  $n(0) = n_0$ and
  \begin{equation}
    \label{soln-linear}
    n(t,s) = T_t n_0(s)
    - \int_0^t T_{t-\tau} \big(p (\cdot)
    n(\tau,\cdot) \big) (s) \d \tau
    + \int_0^t T_{t-\tau} \big( N(\tau) \delta_0 \big) (s) \d \tau
  \end{equation}
  for all $t \in [0,T)$, with
  \begin{equation}
    \label{soln-linear-N}
    N(t) := \int_0^\infty p(s) n(t,s) \d s,
    \qquad t \in [0, T).
  \end{equation}
\end{dfn}

Our existence result stated in \ref{thmwellposed} easily gives the following as
a consequence:

\begin{thm}[Well-posedness of \eqref{eq:nofatigue-linear} in measures]
  \label{thm:wellposed-linear}
  Assume that $p \: [0,+\infty) \to [0,+\infty)$ is bounded, Lipschitz
  and nonnegative. For any given initial data
  $n_0 \in \mathcal{M}(\mathbb{R}^+)$ there exists a unique measure
  solution $n \in \mathcal{C}([0,+\infty); \mathcal{M}(\R^+_0))$ of
  the linear equation \eqref{eq:nofatigue-linear} in the sense of
  Definition \ref{defnmeasuresoln-linear}. In addition, if $n$ is any
  mild measure solution to \eqref{eq:nofatigue-linear} defined on any
  interval $[0,T)$ then
  \begin{equation}
    \label{eq:t-stability}
    \| n(t) \|_{\mathrm{TV}}
    \leq
    \| n(0) \|_{\mathrm{TV}}
    \qquad
    \text{for all $t \in [0,T)$.}
  \end{equation}
\end{thm}

\begin{proof}
  This result can be mostly deduced from Theorem
  \ref{thmwellposed}. For the existence part, split $n_0$ into its
  positive and negative parts as $n_0 = n_0^+ - n_0^-$. Theorem
  \ref{thmwellposed} gives the existence of two solutions $n^+$ and
  $n^-$ with initial data $n_0^+$ and $n_0^-$, respectively; then
  $n := n^+ - n^-$ is a mild measure solution with initial data $n_0$.
  For uniqueness, if $n$ is any mild solution on $[0,T)$, the same
  argument as in \eqref{eq:stab-prf-1}--\eqref{eq:stab-prf-2} shows
  that
  \begin{equation*}
    \|n(t)\|_{\mathrm{TV}}
    \leq
    \| n_0 \|_{\mathrm{TV}}
    + 4 \|p\|_{\infty}
    \int_0^t \|n_1(\tau) \|_{\mathrm{TV}} \d \tau,
  \end{equation*}
  which gives by Gronwall's inequality that 
  \begin{equation*}
    \| n(t) \|_{\mathrm{TV}}
    \leq
    \| n_0 \|_{\mathrm{TV}}
    \,
    e^{ 4 \|p\|_{\infty} t }
    \qquad
    \text{for all $t \in [0,T)$.}
  \end{equation*}
  In particular, by linearity this implies solutions are
  unique. Finally, with the same argument as in Lemma
  \ref{lem:mass-conservation} one sees that for any solution $n$
  defined on $[0,T)$ it holds
  \begin{equation*}
    \int_{\R^+_0} n(t,s) \d s = \int_{\R^+_0} n_0(s) \d s
  \end{equation*}
  for all $t \in [0,T)$. Due to uniqueness, with the same splitting we
  used at the beginning of the proof we have $n(t) = n^+(t) - n^-(t)$,
  so
  \begin{equation*}
    \|n(t)\|_{\mathrm{TV}}
    \leq
    \|n^+(t)\|_{\mathrm{TV}}  +    \|n^-(t)\|_{\mathrm{TV}}
    =
    \|n^+_0\|_{\mathrm{TV}}  +    \|n^-_0\|_{\mathrm{TV}}
    = \|n_0\|_{\mathrm{TV}},
  \end{equation*}
  which finishes the proof.
\end{proof}

The above result allows us to define an evolution semigroup
$(S_t)_{t \geq 0}$ (in fact it is a $C_0$-semigroup on $\mathcal{M}$
with the bounded Lipschitz topology) by setting
\begin{equation*}
  S_t \: \mathcal{M} \to \mathcal{M},
  \qquad S_t(n_0) := n(t)
\end{equation*}
for any $n_0 \in \mathcal{M}$, where $n(t)$ is the mild measure
solution to \eqref{eq:nofatigue-linear} with initial data $n_0$.

\paragraph{Stationary solutions for the linear equation}

We remark that Theorem \ref{thm:ssfornofatigue} below implies that the
linear equation \eqref{eq:nofatigue-linear} has a unique stationary
solution in the space of probabilities on $[0,+\infty)$ (for $p$
bounded, Lipschitz, and satisfying \eqref{p4}); of course in this case
this solution is explicit, given by
\begin{equation*}
  n_*(s) := N_* e^{-\int_0^s p(\tau) \d \tau},
  \qquad s \geq 0,
\end{equation*}
where $N_*$ is the appropriate normalisation constant that makes this
a probability density. Although Theorem \ref{thm:ssfornofatigue} does
not rule out the existence of other stationary solutions which may not
be probabilities, since the solution is explicit it is not difficult
to see that, up to a constant factor, $n_*$ is the only stationary
solution within the set of all finite measures. This is also a
consequence of Doeblin's theorem below.

\subsubsection{Positive lower bound}

Our main result on the spectral gap for the linear operator is based
on the fact that for any initial probability distribution, solutions
have a universal lower bound after a fixed time. We give the following
lemma:
\begin{lem}
  \label{lem-Doeblin1}
  Let $p \:[0,+\infty) \to [0,+\infty)$ be bounded, Lipschitz function
  satisfying \eqref{p3} and \eqref{p4}, and consider the semigroup
  $(S_t)_{t \geq 0}$ given by the existence Theorem
  \ref{thm:wellposed-linear}. Then $S_{t_0}$ satisfies Doeblin's
  condition \eqref{eq:Doeblin} for $t_0 = 2s_*$ and $\alpha = p_{\min}s_* e^{- 2p_{\max} s_*}$. More precisely, for
  $t_0 = 2 s_*$ we have
  \begin{equation*}
    S_{2s_*} n_0(s) \geq  p_{\min} e^{- 2p_{\max} s_*} \1_{\{ 0 < s < s_* \}}
  \end{equation*}
  for all probability measures $n_0$ on $[0,+\infty)$.
\end{lem}

\begin{proof}
  We define a semigroup $\tilde{S}_t$ associated to the linear
  problem \begin{align} \label{eq:tilden}
            \begin{split}
              \begin{cases}
                \frac{\p}{\p t}\tilde{n}(t,s) + \frac{\p}{\p s} \tilde{n}(t,s) = -p(s) \tilde{n}(t,s), \quad t,s >0, \\
                \tilde{n}(t,0) = 0, \quad t>0\\
                \tilde{n}(0,s) = n_0(s), \quad s \geq 0,
              \end{cases}
            \end{split}
          \end{align} which has the explicit solution
          \begin{align*}
            \tilde{n}(t,s) =
            \begin{split}
              \begin{cases}
                n_0(s-t) e^{-\int_{0}^{t} p(s-t +\tau) \d \tau}, \qquad &s>t,\\
                0, \qquad &t>s.
              \end{cases}
	\end{split}
	\end{align*}
	Then we write the solution to the linear equation \eqref{eq:nofatigue-linear} as
        \begin{align*}
n(t,s) = \tilde{S}_t n_0(s) + \int_{0}^{t} \tilde{S}_{t-\tau} (N(\tau)\delta_0 )(s) \d \tau.
	\end{align*}
	For $s >t$ we have
        \begin{align*}
          n(t,s) \geq \tilde{S}_t n_0(s) &= n_0(s-t) e^{-\int_{0}^{t} p(s-t +\tau) \d \tau} \geq  n_0(s-t)e^{-p_{\max}t}\\
          \tilde{S}_{t-\tau} n_0(s) &\geq n_0(s-t+\tau) e^{-p_{\max} (t-\tau)}.
	\end{align*}
	Then for $t > s_*$ it holds that
        \begin{multline*}
          N(t) =  \int_{0}^{+\infty} p(s)n(t,s)\d s \geq p_{\min}e^{-p_{\max}t} \int_{s_*}^{\infty} n_0(s-t) \d s  \\ \geq p_{\min}e^{-p_{\max}t} \int_{0}^{+\infty} n_0(s) \d s  =  p_{\min}e^{-p_{\max}t}. 
	\end{multline*}
        Therefore, for any $s > 0$ and $t > s + s_*$ we have:
        \begin{multline*}
          n(t,s) \geq  \int_{0}^{t} \tilde{S}_{t-\tau} (N(\tau)\delta_0 )(s) \d \tau 
          \geq
          \int_{s_*}^{t} \tilde{S}_{t-\tau} (p_{\min}e^{-p_{\max} \tau} \delta_0 )(s) \d \tau
          \\ \geq p_{\min} \int_{s_*}^t e^{-p_{\max} \tau} e^{-p_{\max}(t-\tau)}
          \delta_0(s-t+\tau) \d \tau
          =  p_{\min} e^{-p_{\max} t} \1_{ \{0 < s < t- s_*\} }.
        \end{multline*}
	Hence for $t = 2 s_*$ and all $0 < s < s_*$ we obtain the result.
\end{proof}

\subsubsection{Spectral gap}

Exponential convergence to the equilibrium for the linear equation is an
immediate consequence of Theorem \ref{thm:Doeblin-semigroup}. We give the following proposition based on that:

\begin{prp}
  \label{prop2}
  For a given initial data $n_0 \in \mathcal{M}(\R^+)$, let
  $p \: [0,\infty) \to [0,+\infty)$ be bounded, Lipschitz function 
  satisfying \eqref{p3} and \eqref{p4}. Then, there exists a unique probability measure
  $n_* \in \mathcal{P}([0,+\infty))$ which is a stationary solution to
  \eqref{eq:nofatigue-linear}, and any other stationary solution is a
  multiple of it. Also, for \begin{equation*}
 C =\frac{1}{1-\alpha} >1 \text{ and }  \lambda := - \frac{\log (1-\alpha)}{t_0}
  \end{equation*} we have
  \begin{equation}
    \| S_{t} (n_0 - n_* )\|_{\mathrm{TV}}
    \leq
    Ce^{-\lambda t} \| n_0 - n_* \|_{\mathrm{TV}}, \text{ for all } t \geq 0.
  \end{equation}
  In addition, for $t_0 := 2 s_*$ we have
  \begin{equation}
  \label{eq:St0-contractive-nofatigue}
  \| S_{t_0} (n_1 - n_2) \|_{\mathrm{TV}}
  \leq
  (1-\alpha) \|n_1 - n_2\|_{\mathrm{TV}}
  \end{equation}
  for any probability distributions $n_1, n_2$, and with
  \begin{equation*}
  \alpha :=  p_{\mathrm{min}} s_*  e^{-2 p_{\mathrm{max}}
  	s_*}.
  \end{equation*}
\end{prp}

\begin{proof}
  We apply Theorem \ref{thm:Doeblin-semigroup},
  since Lemma \ref{lem-Doeblin1} shows that $S_{t_0}$ satisfies the
  Doeblin condition for $t_0 = 2s_*$. Moreover, \begin{align*}
  C = \frac{1}{1-\alpha }  = e^{\lambda t_0} = e^{- \log(1-p_{min}s_*e^{-2p_{\max}s_*})} =
  \frac{1}{1-p_{min}s_*e^{-2p_{\max}s_*}} >1. 
  \end{align*}
\end{proof}

\subsection{Stationary solutions for the nonlinear equation}

\begin{dfn}
	We say that a nonnegative function
	$n_* \in \mathcal{C}([0,+\infty)) \cap \mathcal{C}^1(0,+\infty)$ is
	a stationary solution to \eqref{eq:nofatigue} if it satisfies
	\begin{align}\label{eq:nofatss}
	\begin{cases}
	\begin{split}
	&\frac{\p}{\p s} n_*(s) + p(N_*,s)n_*(s) =0, \quad
	s >0,
	\\
	&n_*(0) =: N_* = \int_{0}^{+\infty} p(N_*,s)n_*(s) \d s.
	\end{split}
	\end{cases}
	\end{align}
\end{dfn}

The following result is essentially the same as that given in
\cite{PPS10}. There it is proved for a particular form of $p$, i.e. for
$p(x,s) = \1_{s > s^*(x)}$, for some nonnegative
$s^* \in C^1([0,+\infty))$ such that $\frac{d}{d x} s^*(x) \leq 0$
with $s^*(0) < 1$. So we
prove it here for completeness, and to adapt it to our precise
assumptions:

\begin{thm}
  \label{thm:ssfornofatigue}
  Assume \eqref{p1}, \eqref{p2}, \eqref{p4} for $p$ and also that
  \begin{equation*}
    L <   (p_{\max})^{-2} \left(\frac{s_* ^2}{2} + \frac{s_*}{p_{\min}} +
      \frac{1}{p_{\min}^ 2} \right)^{-1}.
  \end{equation*}
  Then there exists a unique probability measure $n_*$ which is a
  stationary solution to \eqref{eq:nofatigue}.
\end{thm}

\begin{proof}
  If there is a stationary solution $n_*$ then (since the first
  equation of \eqref{eq:nofatss} is an ordinary differential equation)
  it must satisfy
  \begin{equation}
    \label{eq:stationarysoln0}
    n_*(s) = n_*(0) e ^{-\int_{0}^{s}p(N_*, \tau) \d \tau}
    = N_* e ^{-\int_{0}^{s}p(N_*, \tau) \d \tau},
    \qquad s \geq 0.
  \end{equation}
  If $n_*$ is a probability, by integrating we see that
  \begin{equation}
    \label{eq:stationarysoln}
    N_* =
    \left(
      \int_{0}^{+\infty} e^{-\int_{0}^{s}p(N_*,\tau) \d \tau} \d s
    \right)^{-1}.
  \end{equation}
  In particular, $N_*$ must be strictly positive. Conversely, if
  $N_* > 0$ is such that \eqref{eq:stationarysoln} is satisfied
  then we may define $n_* = n_*(s)$ by
  \eqref{eq:stationarysoln0} and it is straightforward to check
  that it is a probability, and it is a stationary solution to
  \eqref{eq:nofatigue}. Hence the problem is reduced to showing
  that there exists a unique solution $N_* > 0$ to
  \eqref{eq:stationarysoln}; this is ensured by a simple fixed
  point argument, since
  \begin{multline*}
    \frac{\p}{\p N} \left(\int_{0}^{+\infty}
      e^{-\int_{0}^{s}p(N,\tau) \d \tau} \d s
    \right)^{-1}
    =
    \frac{\int_{0}^{+\infty}
      \big( \int_{0}^{s} \p_{N} p(N,\tau)d \tau\big)
      \big(e^{-\int_{0}^{s}p(N,\tau) \d \tau} \big) \d s}
    {\big(\int_{0}^{+\infty}
      e^{-\int_{0}^{s}p(N,\tau) \d \tau} \d s \big)^{2}}
    \\
    \leq
    L \frac{\int_{0}^{+\infty}
      s e^{-\int_{0}^{s}p(N,\tau) \d \tau}   \d s }{\Big(\int_{0}^{+\infty}e^{-p_{\max} s} \d s\Big)^2} 
    \leq
    L (p_{\max})^{2}\left(\int_{0}^{s_*} s \d s + \int_{s_*}^{+\infty}
      s e^{-p_{\min} (s-s_*)} \d s \right)
    \\
    =
    L  (p_{\max})^{2} \left(\frac{s_* ^2}{2} + \frac{s_*}{p_{\min}} +
      \frac{1}{p_{\min}^ 2} \right)
    < 1.
    \\
  \end{multline*}
  where we have used \eqref{p2} and \eqref{p4}. Note that these
  calculation is rigorous due to \eqref{p1} and the fact that the
  integrals in $s$ converge uniformly for all $N$.
\end{proof}

Similarly to our main results, the condition on $L$ in the above
theorem can be understood as a condition of weak nonlinearity.

\subsection{Asymptotic behaviour}

\label{sec:asymp1}

In this section we prove Theorem \ref{thm:main} for equation
\eqref{eq:nofatigue}. Formally, the proof is based on rewriting
it as
\begin{equation}
  \label{eq:formal}
  \frac{\p}{\p t} n = \mathcal{L}_N (n) = \mathcal{L}_{N_*} (n) + (\mathcal{L}_N (n) - \mathcal{L}_{N_*} (n))
  =: \mathcal{L}_{N_*} (n) + h,
\end{equation}
where we define
\begin{equation*}
  \mathcal{L}_N(n) (t,s) := -\frac{\p}{\p s} n(t,s) - p(N(t), s) n(t,s)
  + \delta_0(s) \int_0^\infty p(N(t), u) n(t,u) \d u ,
\end{equation*} and
\begin{multline}
\label{eq:h}
h(t,s) := \Big [ p(N_*,s) - p(N(t),s) \Big ] n(t,s)
\\
+ \delta_{0}(s) \int_{0}^{+\infty} \Big [ p(N(t),u)-p(N_*,u)\Big ]
n(t,u) \d u.
\end{multline}
We treat the term $h$ as a perturbation. In order to do this
rigorously, notice that $h$ contains a multiple of $\delta_0$, so it
is necessary to use a concept of solution in a space of
measures. Then, since the solutions we are using do not allow us to
write \eqref{eq:formal} rigorously, we need to use a concept of
solution that allows for the same formal computation; this is the
reason why weak solutions were introduced earlier.

Before proving the Theorem \ref{thm:main} for equation
\eqref{eq:nofatigue} we need the following lemma:

\begin{lem}
  \label{lem:h1}
  Assume the conditions in Theorem \ref{thm:main} for equation
  \eqref{eq:nofatigue}. Then $h$, defined by \eqref{eq:h}, satisfies
  \begin{equation} 
    \| h(t) \|_{\mathrm{TV}}
    \leq \tilde{C} \| n(t) -n_*\|_{\mathrm{TV}}
    \qquad \text{for all $t \geq 0$,}
  \end{equation}
  where $\tilde{C} := 2 p_{\max} \frac{ L }{1-  L }$. It also
  satisfies
  \begin{equation*}
    \int_0^\infty h(t,s) \d s = 0     \qquad \text{for all $t \geq 0$.}
  \end{equation*}
\end{lem}

\begin{proof}
  We notice that the stationary solution $n_*$ exists due to Theorem
  \ref{thm:ssfornofatigue}, and the solution $n(t) \equiv n(t,s)$ with
  initial data $n_0$ was obtained in Theorem \ref{thmwellposed}. Call
  $N_*$ the total firing rate corresponding to the stationary solution
  $n_*$. We estimate directly each of the terms in the expression of
  $h$:
  \begin{multline*}
    \| h(t) \|_{\mathrm{TV}}
    \leq  \| (p(N_*,s) - p(N(t),s)) n(t,s) \|_{\mathrm{TV}}
    + \Big  \| \delta_{0} \int_{0}^{+\infty}
    (p(N(t),s)-p(N_*,s))n(t,s) \d s \Big \|_{\mathrm{TV}}
    \\
    \leq
    \|p(N_*,s) - p(N(t),s)\|_{\infty} \| n(t)\|_{\mathrm{TV}}
    +  \Big | \int_{0}^{+\infty} (p(N(t),s)-p(N_*,s)) n(t,s) \d s
    \Big |
    \\
    \leq
    L |N_*-N(t)|
    +  \|p(N_*,s) - p(N(t),s)\|_{\infty} \| n (t)\|_{\mathrm{TV}}
    \\ 
    \leq
    \frac{L p_{\max}  }{1-  L } \| n(t) -n_*\|_{\mathrm{TV}}
    + L |N_*-N(t)|
    \leq
    2 p_{\max}\frac{L   }{1-  L } \| n(t) -n_*\|_{\mathrm{TV}},
  \end{multline*}
  where the last inequality is due to Lemma \ref{lem:N-stability} and the fact that
  $\|n_* \|_{\mathrm{TV}} = \|n(t) \|_{\mathrm{TV}}=1$, which imply
  \begin{equation*}
    |N_* - N(t)| \leq \frac{p_{\max} }{1-  L } \| n(t) -n_*\|_{\mathrm{TV}}.
  \end{equation*}
  Regarding the integral of $h$ in $s$ we have
  \begin{align*}
    \int_{0}^{+\infty}h(t,s)\d s
    &= \int_{0}^{+\infty} [p(N_*,s) - p(N(t),s)]n(t,s) \d s
    \\&+ \int_{0}^{+\infty} \delta_{0}(x)
        \int_{0}^{+\infty} [p(N(t),s)-p(N_*,s)] n(t,s)\d s \d x
    \\
    &= \int_{0}^{+\infty} [p(N_*,s) - p(N(t),s)]n(t,s) \d s +
      \int_{0}^{+\infty} [p(N(t),s)-p(N_*,s)] n(t,s)\d s
    \\&= 0,
  \end{align*}
  which gives the result.
\end{proof} 

\begin{proof}[Proof of Theorem \ref{thm:main} for
  eq. \eqref{eq:nofatigue}]
  
  Call $N_*$ the value of the total firing rate at equilibrium. The
  solution $n$ to equation \eqref{eq:nofatigue} is in particular a
  weak solution (see Theorem \ref{thm:mild-weak}). Then one sees it is
  also a weak solution (in the sense of \cite{Ball1977}) to the
  equation
  \begin{equation*}
    \ddt n (t,\cdot) = \mathcal{L}_{N_*} n(t,\cdot) + h(t,\cdot),
  \end{equation*}
  where $\mathcal{L}_{N_*}$ is the linear operator corresponding to $p = p(N_*, s)$
  for $N_*$ fixed,
  \begin{align*}
    \mathcal{L}_{N_*} n(t,s)
    := -\frac{\p}{\p s} n(t,s) - p(N_*,s)n(t,s)
    + \delta_{0} \int_{0}^{+\infty} p(N_*,u)n(t,u) \d u.
  \end{align*}
  Then by \cite{Ball1977} we may use Duhamel's formula and write the
  solution as
  \begin{equation}
    n(t,s) = S_t n_0(s) +
    \int_{0}^{t}S_{t-\tau} h(\tau, s) \d \tau,
  \end{equation}
  where $S_t$ is the linear semigroup defined in Section
  \ref{sec:linear}. We subtract the stationary solution from both
  sides;
  \begin{align*}
	n(t,s) - n_*(s)  = S_t n_0(s) - n_*(s)  +  \int_{0}^{t}S_{t-\tau} h(\tau, s) \d \tau.
	\end{align*}
	Then we take the $\mathrm{TV}$ norm;
	\begin{align} \label{eq:duhamelnofatigue}
	\|n(t) - n_* \|_{\mathrm{TV}}  \leq \|  S_t n_0 - n_* \|_{\mathrm{TV}}  +  \Big \| \int_{0}^{t}S_{t-\tau} h(\tau, s) \d \tau \Big \|_{\mathrm{TV}}. 
	\end{align}By using Lemma \ref{lem:h1} and Proposition \ref{prop2}, Equation \eqref{eq:duhamelnofatigue} becomes:
	\begin{align*}
	\|n(t) - n_* \|_{\mathrm{TV}}  &\leq \|  S_t (n_0 - n_*) \|_{\mathrm{TV}}   + \int_{0}^{t} \|  S_{t-\tau} h(\tau, s) \|_{\mathrm{TV}}  \d \tau \\  &\leq C e^{-\lambda t}\| n_0 - n_* \|_{\mathrm{TV}}  +    \tilde{C} \int_{0}^{t}e^{-\lambda (t-\tau)}  \| n(\tau) -n_* \|_{\mathrm{TV}} \d \tau. 
	\end{align*}
Therefore, by Gronwall's inequality we obtain
	\begin{equation*}
          \|n(t) - n_* \|_{\mathrm{TV}}
          \leq
          C  e^{ -(\lambda -\tilde{C})t}  \| n_0 - n_*
          \|_{\mathrm{TV}}.
          \qedhere
	\end{equation*}
\end{proof}

\section{A structured neuron population model with fatigue}

\label{sec:fatigue}

We now consider Equation \eqref{eq:fatnonlin} for a structured neuron
population model. We follow the same order as in Section
\ref{sec:nofatigue}.

\subsection{Well-posedness}

We refer the reader to Section \ref{sec:well-posedness} for
preliminary notation and useful results. We define mild measure
solutions in a similar way. Still denoting by $(T_t)_{t \geq 0}$ the
translation semigroup generated on
$(\mathcal{M}, \|\cdot\|_{\mathrm{BL}})$ by the operator
$-\frac{\p}{\p s}$, we rewrite \eqref{eq:fatnonlin} as
\begin{equation*}
  \frac{\p}{\p t} n(t,s) - \mathcal{L} n(t,s) = A[n](t,s),
\end{equation*}
where
\begin{multline} \label{defnA}
  \mathcal{L} = -\frac{\p}{\p s}
  \quad \text{and}
  \quad A[n](t,s) : = -p(N(t),s)n(t,s)
  + \int_{0}^{+\infty} \kappa (s,u)p(N(t),u)n(t,u) \d u.
\end{multline}

\begin{dfn}
	\label{defnmeasuresolnfat}
	Assume that $p$ satisfies \eqref{p1},\eqref{p2} and $\kappa$ satisfies \eqref{k5}.  A couple of
	functions $n \in \mathcal{C}([0, T), \mathcal{M}_+(\R^+_0))$
	and $N \in \mathcal{C}([0,T), [0,+\infty))$, defined on an interval
	$[0, T)$ for some $T \in (0,+\infty)$, is called a \emph{mild
		measure solution} to \eqref{eq:fatnonlin} with initial data
	$n_0(s) \in \mathcal{M}(\R^+_0)$, $n(0) = n_0$ if it
	satisfies 
	\begin{equation}
	\label{soln2}
	n(t,s) = T_t n_0(s)
	+ \int_0^t T_{t-\tau} A[n(\tau, \cdot)] (s)d\tau, \quad \text{for all }t \in [0,T),
	\end{equation} where $A[n](t,s)$ is defined as in \eqref{defnA} and \begin{equation*}
	N(t) = \int_{0}^{+\infty} p(N(t),s)n(t,s)ds, \quad t \in [0,T).
	\end{equation*}
\end{dfn}
By integrating in $\R^+_0$, Definition \ref{defnmeasuresolnfat} directly
implies mass conservation. Therefore Lemma \ref{lem:mass-conservation} holds true for this equation as well. Moreover we have the Lemma \ref{lem:N-stability} satisfied with the same constants.

\begin{thm}[Well-posedness of \eqref{eq:fatnonlin} in measures]
	\label{thmwellposedfat}
	Assume that $p$ satisfies
	\eqref{p1}, \eqref{p2} and the Lipschitz
	constant $L$ in \eqref{p2} satisfies
	$L \leq 1/(4 \|n_0\|_{\mathrm{TV}})$. Assume also \eqref{k5} for $\kappa$. For any given initial data
	$n_0 \in \mathcal{M}(\mathbb{R}^+)$ there exists a unique measure
	solution $n \in \mathcal{C}([0,+\infty); \mathcal{M}(\R^+_0))$ of
	\eqref{eq:fatnonlin} in the sense of Definition
	\ref{defnmeasuresolnfat}. In addition, if $n_1$, $n_2$ are any two mild
	measure solutions to \eqref{eq:nofatigue} (with possibly different
	initial data) defined on any interval $[0,T)$ then
	\begin{equation}
	\label{eq:t-stability2}
	\| n_1(t) - n_2(t) \|_{\mathrm{TV}}
	\leq
	\| n_1(0) - n_2(0) \|_{\mathrm{TV}}
	\,
	e^{ 4 \|p\|_{\infty} t }
	\qquad
	\text{for all $t \in [0,T)$.}
	\end{equation}
\end{thm}

\begin{proof}[Proof of Theorem \ref{thmwellposedfat}]
	Let us first show existence of a solution for a nonnegative initial
	measure $n_0 \in \mathcal{M}_+(\R^+_0)$. If $n_0 = 0$ it is clear that
	setting $n(t)$ equal to the zero measure on $\R^+_0$ for all $t$
	defines a solution, so we assume $n_0 \neq 0$. Fix $C, T > 0$, to be
	chosen later. Consider the complete metric space
	\begin{equation*}
	\mathcal{Y}=
	\{ n \in \mathcal{C}([0,T], \mathcal{M}_+(\R^+_0)) \mid
	n(0) = n_0,
	\
	\|n(t) \|_{\mathrm{TV}} \leq C \text{ for all $t \in [0,T]$} \},
	\end{equation*}
	endowed with the norm
	\begin{equation*}
	\|n\|_{\mathcal{Y}} := \sup_{t \in [0,T]} \|n(t)\|_{\mathrm{TV}}.
	\end{equation*}
	We remark that $\mathcal{C}([0,T], \mathcal{M}(\R^+_0))$ refers to
	functions which are continuous in the bounded Lipschitz topology,
	\emph{not} in the total variation one. Define an operator
	$\Gamma : \mathcal{Y} \to \mathcal{Y}$ by
	\begin{align}
	\label{bp}
	\Gamma[n](t):= T_t n_0
	+ \int_0^t T_{t-\tau} (A[n](\tau,,))(s) d\tau
	\end{align}
	for all $n \in \mathcal{Y}$. \\
	
	The definition of $\Gamma[n]$ indeed makes sense, since
	$\tau \mapsto T_{t-\tau} (A[n](\tau,\cdot))$ is a
	continuous function from $[0,T]$ to $\mathcal{M}(\R^+_0)$, hence
	integrable.
	
	We first check that $\Gamma[n]$ is indeed in $\mathcal{Y}$. It is easy
	to see that $t \mapsto \Gamma[n](t)$ is continuous in the bounded
	Lipschitz topology, and it is a nonnegative measure for each
	$t \in [0,T]$. We also have
	\begin{multline*}
	\|\Gamma[n](t)\|_{\mathrm{TV}}
	\leq
	\|n_0\|_{\mathrm{TV}}
	+
	\int_0^t \| p (N(\tau), \cdot)
	n(\tau,\cdot) \|_{\mathrm{TV}} \d \tau 
	\\ + \int_0^t  \bigg \| \int_{0}^{+\infty}  \kappa (.,u) p(N(\tau),u) n(\tau, u) \d u \bigg \|_{\mathrm{TV}} \d \tau
	\\
	\leq
	\|n_0\|_{\mathrm{TV}}
	+
	T C \|p\|_{\infty} 
	\\ + \int_0^t  \int_{0}^{+\infty}  \bigg | \Big( \int_{0}^{+\infty}   \kappa (s,u) \d s \Big) p(N(\tau),u) n(\tau, u) \bigg |   \d u   \d \tau \\ 
	\leq
	\|n_0\|_{\mathrm{TV}}
	+ T C \| p \|_{\infty} + \int_0^t \| p (N(\tau), \cdot)
	n(\tau,\cdot) \|_{\mathrm{TV}} \d \tau  \\ \leq  \|n_0\|_{\mathrm{TV}} +  2 T C \|p\|_{\infty} 
	\end{multline*} Here we used the assumption \eqref{k5} on $\kappa$. We choose
	\begin{equation}
	\label{eq:rt1fat}
	T \leq \frac{1}{4 \|p\|_{\infty}}, \quad \text{and} \quad 
	C := 2 \|n_0\|_{\mathrm{TV}},
	\end{equation}
	so that
	\begin{equation*}
	\|n_0\|_{\mathrm{TV}}
	+
	2 T C \| p \|_{\infty}
	\leq
	\|n_0\|_{\mathrm{TV}}
	+
	\frac{C}{2}
	\leq
	C.
	\end{equation*}
	Hence with these conditions on $T$ and $C$ we have
	$\Gamma[n] \in \mathcal{Y}$.
	
	Let us show that $\Gamma$ is a contraction mapping. Take
        $n_1, n_2 \in \mathcal{Y}$ and let $N_1, N_2$ be defined by
        \eqref{eq:N-from-n} corresponding to $n_1$ and $n_2$,
        respectively. We have
	\begin{multline*}
          \| \Gamma[n_1](t) - \Gamma[n_2](t) \|_{\mathrm{TV}}
          \leq
          \int_0^t \| (p (N_1(\tau), \cdot) - p(N_2(\tau),\cdot))
          n_1(\tau,\cdot)  \|_{\mathrm{TV}} \d \tau
          \\
          + \int_0^t
          \| p (N_2(\tau), \cdot)
          (n_1(\tau, \cdot) - n_2(\tau, \cdot)) \|_{\mathrm{TV}} \d \tau
          \\ 
          + \int_0^t
          \bigg \|
          \int_{0}^{+\infty} \kappa (\cdot,u)
          (p (N_1(\tau), u) - p(N_2(\tau), u))
          n_1(\tau, u) \d u
          \bigg \|_{\mathrm{TV}}  \d \tau 
          \\
          + \int_0^t \bigg \|
          \int_{0}^{+\infty}
          \kappa (\cdot,u) p (N_2(\tau), u)
          (n_1(\tau,u) - n_2(\tau,u)) \d u \bigg  \|_{\mathrm{TV}}
          \d \tau
          \\ 
          =: T_1 + T_2 + T_3 + T_4.
	\end{multline*}
	We bound each term separately. We can bound $T_3$ and $T_4$ in the following way;
	\begin{align*}
          T_3
          &\leq \int_0^t  \int_{0}^{+\infty}
            \bigg |
            \left(
            \int_{0}^{+\infty}
            \kappa (s,u) \d s
            \right)
            (p (N_1(\tau), u) - p(N_2(\tau), u))
            n_1(\tau, u) \bigg | \d u  \d \tau = T_1,
          \\
          T_4
          &\leq \int_0^t  \int_{0}^{+\infty}
            \bigg |
            \left(
            \int_{0}^{+\infty}
            \kappa (s,u) \d s
            \right)
            p (N_2(\tau), \cdot) (n_1(\tau) - n_2(\tau))  \bigg |
            \d u \d \tau = T_2
	\end{align*}
        For $T_1$, since $L \leq 1 / ({4 \|n_0\|_{\mathrm{TV}}})$,
        using Lemma~\ref{lem:N-stability} as previously calculated we
        have
	\begin{equation}
	\label{eq:b1fat}
	T_1 \leq
	T \|p\|_{\infty} \| n_1 - n_2 \|_{\mathcal{Y}},
	\end{equation}
	and for $T_2$,
	\begin{equation}
	\label{eq:b2fat}
	T_2 \leq T \|p\|_{\infty} \|n_1 - n_2\|_{\mathcal{Y}}.
	\end{equation}
	Putting equations \eqref{eq:b1fat}--\eqref{eq:b2fat} together and taking
	the supremum over $0 \leq t \leq T$,
	\begin{equation*}
	\| \Gamma[n_1] - \Gamma[n_2] \|_{\mathcal{Y}}
	\leq
	4 T \|p\|_{\infty} \|n_1 - n_2\|_{\mathcal{Y}}.
	\end{equation*}
	Taking now $T \leq 1 / (4 \|p\|_{\infty})$ ensures that $\Gamma$ is
	contractive, so it has a unique fixed point in $\mathcal{Y}$, which
	is a mild measure solution on $[0,T]$. If we call $n$ this fixed
	point, since $\|n(T)\|_{\mathrm{TV}} = \|n_0\|_{\mathrm{TV}}$ by
	mass conservation (see Lemma \ref{lem:mass-conservation}), we may
	repeat this argument to continue the solution on $[T, 2T]$, $[2T,
	3T]$, showing that there is a solution defined on $[0,+\infty)$.
	
	In order to show stability of solutions with respect to the initial
	data (which implies uniqueness of solutions), take two measures
	$n_0^1, n_0^2 \in \mathcal{M}_+(\R^+_0)$, and consider two solutions
	$n_1$, $n_2$ with initial data $n_0^1$, $n_0^2$ respectively. We
	have
	\begin{multline*}
	\| n_1(t) - n_2(t) \|_{\mathrm{TV}}
	\leq
	\| n_0^1 - n_0^2\|_{\mathrm{TV}}
	+
	\int_0^t \| (p (N_1(\tau), \cdot) - p(N_2(\tau),\cdot))
	n_1(\tau,\cdot)  \|_{\mathrm{TV}} \d \tau
	\\
	+ \int_0^t
	\| p (N_2(\tau), \cdot) (n_1(\tau) - n_2(\tau)) \|_{\mathrm{TV}} \d \tau \\ 
	+ \int_0^t \bigg \| \int_{0}^{+\infty} \kappa (s,u) (p (N_1(\tau), \cdot) - p(N_2(\tau),\cdot))
	n_1(\tau,\cdot) \d u  \bigg \|_{\mathrm{TV}}  \d \tau 
	\\ + \int_0^t \bigg \| \int_{0}^{+\infty} \kappa (s,u) p (N_2(\tau), \cdot) (n_1(\tau) - n_2(\tau)) \d u \bigg  \|_{\mathrm{TV}} \d \tau,
	\end{multline*}
	and with very similar arguments as before we obtain that
	\begin{equation*}
	\| n_1(t) - n_2(t) \|_{\mathrm{TV}}
	\leq \| n_0^1 - n_0^2\|_{\mathrm{TV}}
	+ 4 \|p\|_{\infty}
	\int_0^t \|n_1(\tau) - n_2(\tau) \|_{\mathrm{TV}} \d \tau.
	\end{equation*}
	Gronwall's inequality then implies \eqref{eq:t-stability2}.  
\end{proof}

\subsection{The linear equation}

\label{sec:linear2}

The linear version of equation
\eqref{eq:fatnonlin} obtained when $p = p(N,s)$ does not depend on $N$:
\begin{align} \label{eq:fatlin}
\begin{cases}
\begin{split}
\frac{\p}{\p t} n(t,s) + \frac{\p}{\p s} n(t,s) &+ p(s)n(t,s) = \int_{0}^{+\infty} \kappa(s,u)p(u)n(t,u) \d u, \quad u,s,t > 0, \\
n(t, s=0) &= 0 , \quad N= \int_{0}^{+\infty} p(s)n(t,s) \d s, \quad t >0, \\ 
n(t=0,s) &= n_0(s), \quad s \geq 0.
\end{split}
\end{cases}
\end{align}

\subsubsection{Well-posedness}

\label{sec:linear-well-posed2}

Similarly to Section \ref{sec:linear-well-posed}, we can generalise slightly
our concept of solution to include measures which are not necessarily
nonnegative:
\begin{dfn}
	\label{defnmeasuresolutions-linear2}
	Assume that $p \: [0,+\infty) \to [0,+\infty)$ is a bounded, nonnegative function satisfying \eqref{p2} and $\kappa$ satisfies \eqref{k5}.  A couple of
	functions $n \in \mathcal{C}([0, T), \mathcal{M}_+(\R^+_0))$
	and $N \in \mathcal{C}([0,T), [0,+\infty))$, defined on an interval
	$[0, T)$ for some $T \in (0,+\infty]$, are called a \emph{mild
		measure solution} to \eqref{eq:fatnonlin} with initial data
	$n_0(s) \in \mathcal{M}(\R^+_0)$ if it
	satisfies $n(0) = n_0$
	\begin{equation}
          \label{eq:soln-linear-fatigue}
          n(t,s) = T_t n_0(s)
          + \int_0^t T_{t-\tau} A[n(\tau, \cdot)] (s) \d \tau 
	\end{equation}
	for all $t \in [0,T)$ where 
	\begin{equation*}
          A[n](t,s) : = -p(s)n(t,s)
          + \int_{0}^{+\infty} \kappa (s,u)p(u)n(t,u) \d u
	\end{equation*} and \begin{equation*}
	N(t) = \int_{0}^{+\infty} p(s)n(t,s) \d s, \quad t \in [0,T).
	\end{equation*}
\end{dfn}
By the existence result for \eqref{eq:fatnonlin} we have:

 \begin{thm}[Well-posedness of \eqref{eq:fatlin} in measures]
 	\label{thmwellposedfat-linear}
 	Assume that $p \: [0,+\infty) \to [0,+\infty)$ is a bounded,
        nonnegative, Lipschitz function. Assume also that $\kappa$
        satisfies \eqref{k5}. For any given initial data
        $n_0 \in \mathcal{M}(\mathbb{R}^+)$ there exists a unique
        measure solution
        $n \in \mathcal{C}([0,+\infty); \mathcal{M}(\R^+_0))$ of the
        linear equation \eqref{eq:fatlin} in the sense of Definition
        \ref{defnmeasuresolutions-linear2}. In addition, if $n$ is a
        mild measure solution to \eqref{eq:fatlin} defined on any
        interval $[0,T)$ then
 	\begin{equation}
 	\label{eq:t-stability2lin}
 	\| n(t) \|_{\mathrm{TV}}
 	\leq
 	\| n(0) \|_{\mathrm{TV}},
 	\qquad
 	\text{for all $t \in [0,T)$.}
 	\end{equation}
\end{thm}
 
For the proof of this result one can follow the same procedure as in
the proof of Theorem \ref{thm:wellposed-linear}, so we omit it here.
Theorem \ref{thmwellposedfat-linear} allows us to define a
$C_0$-semigroup $(S_t)_{t\geq 0}$ on $\mathcal{M}$, such that
$S_t(n_0) := n(t)$ for any $n_0 \in \mathcal{M}$ where $n(t)$ is the
mild solution to \eqref{eq:fatlin} similarly as in Section
\ref{sec:linear-well-posed}.

Given $p$, we define $\mathcal{L}$ as the generator of the
corresponding semigroup $S_t$, defined on its domain
$\mathcal{D}(\mathcal{L})$. One can of course see that for
sufficiently regular measures $n$,
\begin{equation}
  \label{eq:generator-explicit}
  \mathcal{L} n(s)
  = \frac{\p}{\p s} n(s) + p(s)n(s) -\int_{0}^{+\infty} \kappa (s,u) p(u)n(u) \d u.
\end{equation}
Since the only unbounded operator involved in this expression is $\frac{\p}{\p s} n$, one sees that the domain $\mathcal{D}(\mathcal{L})$ can be
described explicitly as
\begin{equation*}
  \mathcal{D}(\mathcal{L}) := \Big \{n \in \mathcal{M}(\R^+_0)
  \mid 
  \frac{\p}{\p s} n \in \mathcal{M}(\R^+_0)
  \Big \},
\end{equation*}
where the derivative is taken in the sense of distributions on
$\R$. Expression \eqref{eq:generator-explicit} is valid for all
$n \in \mathcal{D}(\mathcal{L})$, again understanding the derivative
in distributional sense.

Finally, for the arguments regarding the nonlinear equation
\eqref{eq:fatnonlin} we will need a result on continuous dependence of
the solutions of the linear equation \eqref{eq:fatlin} on the firing
rate $p$:

\begin{thm}[Continuous dependence with respect to $p$ for the linear equation]
  Let $p_1, p_2 $ be bounded, nonnegative, Lipschitz functions. Assume
  also that $\kappa$ satisfies \eqref{k5}. For any given initial data
  $n_0 \in \mathcal{M}(\mathbb{R}^+)$ consider $n_1, n_2$ the two
  solutions to the linear equation \eqref{eq:fatlin} on $[0,+\infty)$
  with firing rate $p_1$, $p_2$ respectively and initial data
  $n_0$. Assuming $\|p_1\|_\infty \neq 0$, it holds that
  \begin{equation}
    \label{eq:p-stability}
    \| n_1(t) - n_2(t) \|_{\mathrm{TV}}
    \leq
    \frac{\| n_0 \|_{\mathrm{TV}} \|p_1 - p_2 \|_{\infty}}
    {\|p_1\|_\infty} \big(
    e^{2 \|p_1\|_\infty t} - 1
    \big)
    \qquad
    \text{for all $t \geq 0$.}
  \end{equation}
\end{thm}

\begin{proof}
  With the obvious changes in notation, from
  \eqref{eq:soln-linear-fatigue} we have
  \begin{multline*}
    \|n_1(t) - n_2(t)\|_{\mathrm{TV}} \leq
    \int_0^t \| T_{t-\tau} A_1[n_1(\tau, \cdot)]
    - T_{t-\tau} A_2[n_2(\tau, \cdot)] \|_{\mathrm{TV}} \d \tau 
    \\
    =
    \int_0^t \| A_1[n_1(\tau)] - A_2[n_2(\tau)] \|_{\mathrm{TV}} \d \tau.
  \end{multline*}
  In a very similar way as the estimate we carried out for Theorem
  \ref{thmwellposedfat}, this last term can be estimated as
  \begin{multline*}
    \| A_1[n_1(\tau)] - A_2[n_2(\tau)] \|_{\mathrm{TV}}
    \\
    \leq
    2\|p_1(s) (n_1(\tau,s) - n_2(\tau,s)) \|_{\mathrm{TV}}
    + 2\| n_2(s) (p_1(s) - p_2(s)) \|_{\mathrm{TV}}
    \\
    \leq
    2 \|p_1 \|_\infty \| n_1(\tau,\cdot) - n_2(\tau,\cdot)) \|_{\mathrm{TV}}
    + 2 \| n_0 \|_{\mathrm{TV}} \|p_1 - p_2 \|_{\infty}.
  \end{multline*}
  Hence, calling
  $m(t) \equiv m(t) := \|n_1(t,\cdot) - n_2(t,\cdot)\|_{\mathrm{TV}}$
  and $K := \| n_0 \|_{\mathrm{TV}} \|p_1 - p_2 \|_{\infty}$,
  we have
  \begin{equation*}
    m(t)
    \leq
    2 \|p_1\|_\infty
    \int_0^t m(\tau) \d \tau
    + 2 t K.
  \end{equation*}
  Gronwall's Lemma then shows that
  \begin{equation*}
    m(t) \leq \frac{K}{\|p_1\|_\infty} \big(
    e^{2 \|p_1\|_\infty t} - 1
    \big).
    \qedhere
  \end{equation*}
\end{proof}

\paragraph{Stationary solutions for the linear equation}

\begin{dfn}
  \label{dfn:ssfatlin}
   A stationary solution to \eqref{eq:fatlin} $n_* \in \mathcal{M}$ is defined 
  as such that $n_* \in \mathcal{D}(\mathcal{L})$ and
  \begin{equation*}
    \label{ssfatlin}
      \mathcal{L} n_* = 0.
  \end{equation*}
\end{dfn}
We remark that Proposition \ref{prop4} below implies that the linear
equation \eqref{eq:fatlin} has a unique stationary solution in the
space of probabilities on $[0,+\infty)$ (for $p$ bounded, Lipschitz,
satisfying \eqref{p4} and $\kappa$ satisfying \eqref{k5}); $n_*$ is
the only stationary solution up to a constant factor within the set of
all finite measures.

\subsubsection{Positive lower bound}

Analogous to Section \ref{sec:linear}, we want to show that for a
given positive initial distribution, solutions of \eqref{eq:fatlin}
after some time have a positive lower bound, so that the semigroup
$S_t (n_0)$ satisfies the \textit{Doeblin's condition} given in
\eqref{eq:Doeblin}.
 
\begin{lem} \label{lem-Doeblin2} Let
  $p : [0,+\infty) \longrightarrow [0,+\infty)$ be a bounded,
  Lipschitz function satisfying \eqref{p3} and \eqref{p4}. We assume
  also that $\kappa$ satisfies \eqref{k5} and \eqref{k6}. Consider the
  semigroup defined as $S_t(n_0) := n(t)$ for any
  $n_0 \in \mathcal{M}$. Then $S_{t_0}$ satisfy the Doeblin condition
  \eqref{eq:Doeblin} for $t_0=2s_*$ and $\alpha = \epsilon \delta
  p_{\min}(s_*-\delta) e^{-p_{\max} t_0}$. More precisely, for $t_0=2s_* $
  we have \begin{equation*} S_{2s_*} n_0(s) \geq \epsilon \delta
    p_{\min} e^{-2p_{\max} s_*} \1_{\{ \delta < s < s_*\}}
	\end{equation*} for all probability measures $n_0$ on $[0,+\infty)$.
\end{lem}
\begin{proof}
  Since for $s,t > 0$, it holds true for solutions of
  \eqref{eq:fatlin} that
  \begin{equation*} \frac{\p}{\p t} n(t,s) +
    \frac{\p}{\p s} n(t,s) \geq - p(s)n(t,s).
  \end{equation*}
  Moreover, solutions of \eqref{eq:fatlin} satisfy
  $n(t,s) \geq \tilde{n}(t,s)$ where the equation on $\tilde{n}(t,s)$
  was defined in \eqref{eq:tilden} of Lemma \ref{lem-Doeblin1}. By the
  same argument we have for $t > s_*$,
  $N(t) \geq p_{\min} e^{-p_{\max}t}$.

  We consider the same semigroup $\tilde{S}_t$ associated to
  \eqref{eq:tilden}. Then, solutions of \eqref{eq:fatlin}
  satisfy \begin{multline*} n(t,s) = \tilde{S}_t n_0(s) + \int_{0}^{t}
    \tilde{S}_{t-\tau} \Big(
    \int_{0}^{+\infty}\kappa(.,u)p(u)n(t,u) \d u \Big )(s) \d \tau
    \\
    \geq \tilde{S}_t n_0(s) + \int_{0}^{t} \tilde{S}_{t-\tau}
    (\epsilon N(\tau) \1_{ \{ s \leq \delta \} } ) \d \tau
  \end{multline*}  
  since
  \begin{multline*}
    \int_{0}^{+\infty} \kappa(s,u)p(u)n(t,u)\d u
    \geq
    \int_{0}^{+\infty} \phi(s) p(u)n(t,u)\d u
    \geq
    \epsilon \1_{ \{ s \leq \delta \} } \int_{0}^{+\infty}
    p(u)n(t,u)\d u
    \\
    = \epsilon \1_{ \{ s \leq \delta \} } N(t).
  \end{multline*}
  Then for $t > s +s _*$ and $s > \delta$ we have
  \begin{multline*}
    n(t,s) \geq
    \int_{0}^{t} \tilde{S}_{t-\tau}
    (\epsilon N(\tau) \1_{ \{ s \leq \delta \} } ) \d \tau
    \geq
    \int_{s_*}^{t} \tilde{S}_{t-\tau}
    (\epsilon \1_{ \{ s \leq \delta \} } p_{\min} e^{-p_{\max} \tau } ) \d\tau
    \\
    \geq
    \epsilon p_{\min} \int_{s_*}^{t}
    e^{-p_{\max} \tau } e^{-p_{\max} (t-\tau) }
    \1_{ \{ 0 < s -t +\tau \leq \delta \} }
    \d \tau
    =
    \epsilon p_{\min}e^{-p_{\max} t }
    \int_{s_*}^{t} \1_{ \{ 0 < s -t +\tau \leq \delta \} } \d \tau
    \\
    = \epsilon \delta p_{\min} e^{-p_{\max}t}  \1_{\{\delta < s < t_0-s_*\}}.
  \end{multline*}
  Hence for $t=2s_*$ and $\delta < s < s_*$ we obtain the result.
\end{proof}

\subsubsection{Spectral gap} 	

We again obtain a spectral gap property as a consequence of Theorem
\ref{thm:Doeblin-semigroup}:

\begin{prp}
  \label{prop4}
  Let $n_0 \in \mathcal{M}(\mathbb{R}^+)$ be the initial data given
  for \eqref{eq:fatlin}. We assume that $p$ is a nonnegative,
  Lipschitz function satisfying \eqref{p2}--\eqref{p4} and $\kappa$
  satisfies \eqref{k5}, \eqref{k6}. Then, there exists a unique
  probability measure $n_* \in \mathcal{P}([0,+\infty))$ which is a
  stationary solution to \eqref{eq:fatlin}, and any other stationary
  solution is a multiple of it. Also for
  \begin{align*}
    C = \frac{1}{1-\alpha} > 1, \text{ and } \lambda = - \frac{\log(1-\alpha)}{t_0 }, 
  \end{align*} we have 	
  \begin{equation*}
    \| S_{t} (n_0 - n_*) \|_{TV} \leq
    Ce^{-\lambda t} \| n_0 - n_* \|_{\mathrm{TV}}, \text{ for all } t \geq 0.
  \end{equation*}
  In addition, for $t_0 := 2 s_*$ we have
  \begin{equation}
    \label{eq:St0-contractive-fatigue}
    \| S_{t_0} (n_1 - n_2) \|_{\mathrm{TV}}
    \leq
    (1-\alpha) \|n_1 - n_2\|_{\mathrm{TV}}
  \end{equation}
  for any probability distributions $n_1, n_2$, and with
  \begin{equation*}
    \alpha := \epsilon \delta p_{\mathrm{min}} (s_* - \delta) e^{-2 p_{\mathrm{max}}
      s_*} .
  \end{equation*}
\end{prp}

\begin{proof}
  Lemma \ref{lem-Doeblin2} ensures the operator $S_{t_0}$ satisfies
  the \textit{Doeblin condition} \eqref{eq:Doeblin} for $t_0 =
  2s_*$. We obtain the result by applying Theorem
  \ref{thm:Doeblin-semigroup}.
\end{proof}

\subsection{Stationary solutions for the nonlinear equation}

\begin{dfn}
  \label{dfn:ssfatnonlin}
  We say that a pair $(n_*, N_*)$, where
  $n_* \in \mathcal{M}_+(\R^+_0)$ and $N_* \geq 0$, is a
  \emph{stationary solution} to \eqref{eq:fatnonlin} if
  $n_* \in \mathcal{D}(\mathcal{L})$ and
  \begin{equation*}
    \mathcal{L}_{N_*} n_* = 0,
    \qquad
    N_* = \int_{0}^{+\infty} p(N_*,s) n_*(s) \d s,
  \end{equation*}
  where $\mathcal{L}_{N_*}$ is the semigroup generator associated to
  $p(s) \equiv p(N_*, s)$ (see Theorem \ref{thmwellposedfat-linear}
  and the following remarks; observe that the domain
  $\mathcal{D}(\mathcal{L})$ does not depend on the value of $N_*$).
  We say that $N_*$ is the \emph{global neural activity} associated to
  the stationary solution.
\end{dfn}

We give the following theorem for existence and uniqueness of
stationary solutions:

\begin{thm}
  \label{thm:stationaryfatigue}
  Assume \eqref{p1}, \eqref{p2}, \eqref{p4}, \eqref{k5},
  and also that
  $$L < \Big(1+\frac{C p_{\max}}{\alpha p_{\min}} \Big)^{-1},$$ where
  $C := e^{4 p_{\mathrm{max}} s_*}$ and $\alpha$ is given by
  Proposition \ref{prop4}. Then there exists a unique stationary
  solution $(n_*, N_*)$ of \eqref{eq:fatnonlin} such that $n_*$ is a
  probability measure.
\end{thm}

\begin{proof}
  Proposition \ref{prop4} ensures that for a fixed $N_*$, there exists
  a unique probability stationary solution of the corresponding linear
  problem. We prove the existence of a stationary solution by
  recovering $N_*$ from $n_*$ and carrying out a fixed-point argument.
  We define a map
  $\Upsilon: [0, +\infty) \longrightarrow [0, +\infty)$, by
  $$\Upsilon (N) := \int_{0}^{+\infty} p(N,s) n(s)\d s,$$
  where $n$ is the unique probability measure which is an equilibrium
  of the linear problem associated to $p(s) \equiv p(N,s)$. We notice
  that the statement we wish to prove is equivalent to the fact that
  $\Upsilon$ has a unique fixed point.

  Let us show that this map is contractive. For any $N_1, N_2 \geq 0$,
  \begin{multline*}
    | \Upsilon(N_1) - \Upsilon(N_2)|
    = \bigg | \int_{0}^{+\infty}
    \Big( p( N_1,s)n_1(s) -  p(N_2,s)n_2(s) \Big) \d s
    \bigg | 
    \\
    \leq  \int_{0}^{+\infty}| (p(N_1,s) - p(N_2,s)) n_2(s)| \d s
    + \int_{0}^{+\infty} |p(N_1,s) (n_1(s)-n_2(s)) | \d s 
    \\
    \leq L |N_1 -N_2| + p_{\max} \| n_1 -n_2 \|_{\mathrm{TV}}.
  \end{multline*}
  Now, we will prove later that
  \begin{equation}
    \label{eq:1}
    \| n_1 -n_2 \|_{\mathrm{TV}}
    \leq
    \frac{L C}
    {\alpha p_{\mathrm{min}}}
    |N_1 - N_2|,
  \end{equation}
  where $C := e^{4 p_{\mathrm{max}} s_*}$ and $\alpha$ is the one
  from Proposition \ref{prop4}. This implies that
  \begin{equation*}
    | \Upsilon(N_1) - \Upsilon(N_2)|
    \leq L \Big(1+\frac{C p_{\max}}{\alpha p_{\min}} \Big) |N_1-N_2|,
  \end{equation*}
  which makes $\Upsilon$ a contraction operator if $L$ satisfies the
  inequality in the statement. So in order to complete the proof we
  only need to show \eqref{eq:1}. For this we define the two operators
  \begin{equation*}
    P_1 (n) := S^1_{t_0} n - n,
    \qquad
    P_2 (n) := S^2_{t_0} n - n,
  \end{equation*}
  where for $i=1, 2$, $(S^i_t)_{t \geq 0}$ is the linear semigroup
  given by Theorem \ref{thmwellposedfat-linear}, associated to the
  firing rate $p_i(s) := p(N_i, s)$, and $t_0 := 2 s_*$ is the time
  mentioned in Proposition \ref{prop4}. We  use that, since $n_1$,
  $n_2$ are equilibria for the linear equations with $p_1$, $p_2$,
  \begin{equation*}
    0 = P_1 (n_1) = P_2(n_2)
  \end{equation*}
  so that
  \begin{multline*}
    0 = \|P_{1}(n_1) - P_{2}(n_2)\|_{\mathrm{TV}}
    = \| P_{1}(n_1-n_2)
    + (P_{1}-P_{2}) (n_2)\|_{\mathrm{TV}}
    \\
    \geq
    \|P_{1}(n_1-n_2) \|_{\mathrm{TV}}
    - \|(P_{1}-P_{2}) n_2\|_{\mathrm{TV}},
  \end{multline*}
  which implies
  \begin{equation}
    \label{eq:2}
    \|P_{1}(n_1-n_2) \|_{\mathrm{TV}}
    \leq
    \|(P_{1}-P_{2}) n_2\|_{\mathrm{TV}}.
  \end{equation}
  Then by Proposition \ref{prop4} we have
  \begin{align*}
    \|P_{1} (n_1-n_2)  \|_{\mathrm{TV}}
    \geq \|n_1 - n_2\|_{\mathrm{TV}} - \| S_{t_0}^1 (n_1 - n_2) \|_{\mathrm{TV}}
    \geq
    \alpha \| n_1-n_2 \|_{\mathrm{TV}},
  \end{align*}
  since $\int n_1 \d s= \int n_2 \d s = 1$, where $\alpha$ is the one
  in Proposition \ref{prop4}. On the other hand, by
  \eqref{eq:p-stability},
  \begin{multline*}
    \|(P_{1}-P_{2}) n_2\|_{\mathrm{TV}}
    =
    \|(S_{t_0}^1 - S_{t_0}^2) n_2\|_{\mathrm{TV}}
    \leq
    \frac{\| n_2 \|_{\mathrm{TV}} \|p_1 - p_2 \|_{\infty}}
    {\|p_1\|_\infty} \big(
    e^{2 \|p_1\|_\infty t_0} - 1
    \big)
    \\
    \leq
    \frac{L |N_1 - N_2|}
    {p_{\mathrm{min}}}
    e^{2 p_{\mathrm{max}} t_0}.
  \end{multline*}
  Using the last two equations in \eqref{eq:2},
  \begin{equation*}
    \|n_1-n_2\|_{\mathrm{TV}}
    \leq
    \frac{1}{\alpha}
    \|(P_1 - P_2) n_2 \|_{\mathrm{TV}}
    \leq
    \frac{L |N_1 - N_2|}
    {\alpha p_{\mathrm{min}}}
    e^{2 p_{\mathrm{max}} t_0},
  \end{equation*}
  which proves \eqref{eq:1}. Therefore $\Upsilon$ has a unique fixed
  point, and hence \eqref{eq:fatnonlin} has a unique stationary
  solution.
\end{proof} 

\subsection{Asymptotic behaviour}

\label{sec:asymp2}

In this section we prove Theorem \ref{thm:main} for equation \eqref{eq:fatnonlin}. We define two operators in the following way:
\begin{align*}
\mathcal{L}_{N(t)}n(t,s) &:= \p_t n(t,s) = -\p_s n(t,s) -p(N(t),s) n(t,s) + \int \kappa(s,u) p(N(t),u) n(t,u)du,  \\ 
\mathcal{L}_{N_*}\bar{n}(s) &:= -\p_s\bar{n}(s) -p(N_*,s) \bar{n}(s) + \int \kappa(s,u) p(N_*,u) \bar{n}(s)du. 
\end{align*}
We rewrite \eqref{eq:fatnonlin} as \begin{align}
\frac{\p}{\p t}n(t,s) = \mathcal{L}_{N(t)}n(t,s) = \mathcal{L}_{N_*}n(t,s) - (\mathcal{L}_{N_*} - \mathcal{L}_{N(t)})n(t,s).
\end{align}
 Then, similarly as in Section \ref{sec:asymp1} by \cite{Ball1977} we may use Duhamel's formula and write the solution as
\begin{align} \label{eq:ball2}
n(t,s) = S_t n_0(s)  + \int_{0}^{t} S_{t-\tau} h(\tau,s) \d \tau,
\end{align} where $S_t n_0(s) := e^{\mathcal{L}_{N_*}} n_0(s)$ and $\bar{n}$ is the solution to linear problem, $\mathcal{L}_{N_*}$ is acting on $n(t,s)$. Also, 
\begin{multline} \label{eq:h2}
h(t,s) := (\mathcal{L}_{N_*}-\mathcal{L}_{N(t)})n(t,s)  \\ = (p(N(t),s)-p(N_*,s))n(t,s) + \int_{0}^{+\infty} \kappa(s,u)(p(N_*,u)-p(N(t),u))  n(t,u) \d u.  
\end{multline}
 Then we give the following lemma:
 \begin{lem} \label{lem:h2}
 	Assume that \eqref{p2} and \eqref{p4} hold true for a Lipschitz function $p$ and $\kappa$ satisfies \eqref{k5}. Then $h$, which is defined by \eqref{eq:h2}, satisfies
 	\begin{equation} 
 	\| h(t) \|_{\mathrm{TV}} \leq \tilde{C} \| n(t) -n_*\|_{\mathrm{TV}} ,
 	\end{equation}
 	where $\tilde{C} = 2 p_{\max} \frac{ L }{1-  L }$. Moreover $ \int_{0}^{+\infty} h(t,s) \d s =0$.
 \end{lem} 
\begin{proof}
		\begin{multline*}
		\| h(t)\|_{\mathrm{TV}} = \| (\mathcal{L}_{N_*}-\mathcal{L}_{N(t)})n(t,s)\|_{\mathrm{TV}} 
		\\ \leq \|(p(N(t),s)-p(N_*,s))n(t,s) \|_{\mathrm{TV}} + \Big \| \int_{0}^{+\infty} \kappa(s,u)(p(N_*,u)-p(N(t),u))  n(t,u) du \Big \|_{TV} 
		\\  \leq L \| n(t)\|_{\mathrm{TV}} |N_* - N(t)| + L  \| n(t)\|_{\mathrm{TV}} |N_* - N(t)|  
		\\ \leq 2p_{\max}  \frac{ L\| n(t)\|_{\mathrm{TV}} }{1 - L \| n(t)\|_{\mathrm{TV}}} \| n(t) -n_*\|_{\mathrm{TV}} = 2p_{\max}  \frac{ L }{1 - L} \| n(t) -n_*\|_{\mathrm{TV}}
		\end{multline*}
		Since \begin{multline*}
		|N_*-N(t)| = \Big | \int_{0}^{+\infty} p(N_*,s)n_*(s)ds -\int_{0}^{+\infty} p(N(t),s)n(t,s)ds \Big | 
		\\ \leq \Big | \int_{0}^{+\infty} \Big (p(N_*,s)n_*(s) +  (p(N_*,s)n(t,s)-p(N_*,s)n(t,s) ) p(N(t),s)n(t,s) \Big )ds \Big | 
		\\ \leq \Big | \int_{0}^{+\infty}  p(N_*,s)(n_*(s)-n(t,s))ds \Big | + \Big |\int_{0}^{+\infty}  (p(N_*,s)-p(N(t),s)n(t,s)ds\Big | 
		\\ \leq p_{\max} \| n(t) - n_*\|_{\mathrm{TV}} + L |N_*-N(t)| \| n\|_{\mathrm{TV}}  
		\end{multline*} implies that
		\begin{equation}
		|N_*-N(t)| \leq \frac{p_{\max}}{1-L\| n(t)\|_{TV}} \| n(t) - n_*\|_{\mathrm{TV}} = \frac{p_{\max}}{1-L} \| n(t) - n_*\|_{\mathrm{TV}}
		\end{equation}
		since $\| n(t)\|_{\mathrm{TV}} = \| n_*\|_{\mathrm{TV}} =1$.
	Moreover we have \begin{multline*}
	\int_{0}^{+\infty} h(t,s) \d s = \int_{0}^{+\infty} p(N(t),s)\bar{n}(t,s)\d s - \int_{0}^{+\infty} p(N_*,s)\bar{n}(t,s) \d s 
	\\+ \int_{0}^{+\infty} \int_{0}^{+\infty} \kappa(s,u)p(N_*,u)\bar{n}(t,u) \d u\d s - \int_{0}^{+\infty} \int_{0}^{+\infty}  \kappa(s,u) p(N(t),u) \bar{n}(t,u) \d u \d s 
	\\ = N(t) - \int_{0}^{+\infty} p(N_*,s)\bar{n}(t,s) \d s 
	\\+ \int_{0}^{+\infty} \Big(\int_{0}^{u}\kappa(s,u) \d s \Big) p(N_*,u)\bar{n}(t,u)\d u + \int_{0}^{+\infty} \int_{0}^{u}\kappa(s,u) \d s p(N(t),u)\bar{n}(t,u) \d u 
	\\= N(t) - \int_{0}^{+\infty} p(N_*,s)\bar{n}(t,s) \d s + \int_{0}^{+\infty} p(N_*,u)\bar{n}(t,u) \d u - N(t) =0.
	\end{multline*}
\end{proof}

\begin{proof}[Proof of Theorem \ref{thm:main} for \eqref{eq:fatnonlin}]
  We subtract the unique probability stationary solution from both
  sides of \eqref{eq:ball2}:
  \begin{align*}
    n(t,s) - n_*(s) &= S_t n_0 (s) -  n_*(s) + \int_{0}^{t} S_{t-\tau} h(\tau,s) \d \tau.
  \end{align*} 
  We take the total variation norm and obtain
  \begin{align*}
    \|n(t) - n_*\|_{\mathrm{TV}} &\leq\| S_t n_0 -  n_*\|_{\mathrm{TV}} + \Big \| \int_{0}^{t} S_{t-\tau} h(\tau,s) \d \tau \Big \|_{\mathrm{TV}}.
  \end{align*}
  Then by Proposition \ref{prop4} and Lemma \ref{lem:h2} we have
  \begin{multline*}
    \|n(t) - n_*\|_{\mathrm{TV}} \leq C e^{-\lambda t}\| n_0 -  n_*\|_{\mathrm{TV}} + \int_{0}^{t} \|  S_{t-\tau} h(\tau,s) \|_{\mathrm{TV}} \d \tau 
    \\
    \leq C e^{-\lambda t}\| n_0 -  n_*\|_{\mathrm{TV}} + \int_{0}^{t}  e^{-\lambda (t-\tau)}\| h(\tau,s) \|_{\mathrm{TV}} \d \tau 
    \\ \leq C e^{-\lambda t}\| n_0 -  n_*\|_{\mathrm{TV}} + \tilde{C}\int_{0}^{t}   e^{-\lambda (t-\tau)} \| n(\tau)- n_* \|_{\mathrm{TV}} \d \tau.
  \end{multline*}
  Therefore, by Gronwall's inequality
  \begin{equation*}
    \|n(t) - n_*\|_{\mathrm{TV}}
    \leq C e^{-(\lambda-\tilde{C})t} \| n_0-n_*\|_{\mathrm{TV}}.
    \qedhere
  \end{equation*}
\end{proof}

\section*{Acknowledgements}

The authors would like to thank Susana Gutiérrez and Thibault
Bourgeron for several useful discussions on the models studied in this
paper. JAC and HY were supported by projects MTM2014-52056-P and
MTM2017-85067-P, funded by the Spanish government and the European
Regional Development Fund. HY was also supported by the Basque
Government through the BERC 2014-2017 program and by Spanish Ministry
of Economy and Competitiveness MINECO: BCAM Severo Ochoa excellence
accreditation SEV-2013-0323 and by ``la Caixa Grant''.

\bibliography{bibliography}

\end{document}